\documentclass{article}

\usepackage[T1]{fontenc}

\usepackage{amsmath,amssymb,amsthm, xcolor, enumerate, tikz-cd, amsfonts, geometry, mathtools, xspace, authblk, tikz, caption, float,mathabx}
\usepackage[latin1]{inputenc}
\usepackage{graphicx} 
\usepackage[english]{babel}

\usepackage{hyperref}
\hypersetup{
     colorlinks=true,
     urlcolor=violet,
     }
\usepackage{titlesec}
\titleformat{\section}[block]{\small\bfseries\filcenter}{\thesection}{1em}{}
\titleformat{\subsection}[block]{\small\bfseries\filcenter}{\thesubsection}{1em}{}

\newtheorem{thm}{Theorem}[section]
\newtheorem{prop}{Proposition}[section]
\newtheorem{?}{Question}[section]

\newtheorem{lem}{Lemma}[thm]
\newtheorem{claims}{Claim}[thm]

\newtheorem{rmk}{Remark}[section]

\newtheorem{ex}{Example}[section]
\newtheorem{definition}{Definition}[section]
\newtheorem{def-prop}{Definition-Proposition}[section]

\newcommand{\Z}{\mathbb{Z}}

\newcommand{\CP}{\mathbb{C}\mathbb{P}^}
\newcommand{\R}{\mathbb{R}}
\newcommand{\C}{\mathbb{C}}
\newcommand{\h}{\mathbb{H}}

\newcommand{\al}{\alpha}
\newcommand{\dt}{\delta}
\newcommand{\ep}{\epsilon}

\DeclareMathOperator{\Hess}{\text{Hess}}

\usepackage{lipsum}
\title{Examples of Toric Scalar-flat K\"ahler Surfaces with Mixed-type Ends}
\author{Yueqing Feng}
\date{}

\newcommand{\Addresses}{{
 \bigskip
\footnotesize
Yueqing Feng, \textsc{Department of Mathematics, University of California, Berkeley, CA 94720, USA.}\par\nopagebreak
\textit{E-mail address: }\href{fyq@berkeley.edu}{\texttt{fyq@berkeley.edu}}}}

\usepackage{amsrefs}

\begin{document}

\maketitle

\begin{abstract}
    Given a strictly unbounded toric symplectic $4$-manifold, we explicitly construct complete toric scalar-flat K\"ahler metrics on the complement of a toric divisor. These symplectic $4$-manifolds correspond to a specific class of non-compact K\"ahler surfaces. We also provide an alternative construction of toric scalar-flat K\"ahler metrics with conical singularity along the toric divisor, following the approach of Abreu and Sena-Dias.
\end{abstract}
\begingroup
\raggedright 

\section{Introduction}

Let $X$ be a non-compact toric symplectic $4$-manifold. In \cite{AS12}, Abreu and Sena-Dias construct complete toric scalar-flat K\"ahler metrics on strictly unbounded toric symplectic $4$-manifold $X$. We say $X$ is strictly unbounded if the moment polytope of $X$ is unbounded with the unbounded edges being non-parallel. This condition is equivalent to saying there exists a finite sequence of blow-downs of $X$ from which we obtain a minimal resolution of $\C^2\slash\Gamma$ for some finite cyclic subgroup $\Gamma\subset U(2)$. The metrics constructed in \cite{AS12} include the well-known examples of the LeBrun-Simanca metrics \cite{lebrun1991explicit}, the (multi-)Taub-NUT metrics \cite{lebrun1991complete}, the gravitational instantons of Gibbons-Hawking \cite{gibbons1978gravitational} and Kronheimer \cite{kronheimer1993construction}, etc. 

\medskip

In this article, using \textit{Donaldson's ansatz} \cite{Don09} of toric scalar-flat K\"ahelr metrics, we explicitly construct complete toric scalar-flat K\"ahler metrics on the complement of a torus-invariant divisor in $X$, which exhibit \textit{Poincar\'e type} singularity along the divisor. The study of Poincar\'e type K\"ahler metrics stems from the standard Poincar\'e cusp metric

$$\omega_{\Delta^*}=\displaystyle\frac{\displaystyle\sqrt{-1}dz\wedge d\bar{z}}{(|z|\log|z|)^2}=4\displaystyle\sqrt{-1}\partial\bar{\partial}\log(-\log|z|^2)$$

on the punctured unit disk. In general, for a smooth divisor $D$ in $X$, we study complete K\"ahler metrics of Poincar\'e type(see Definition~\ref{definition PT}). Geometrically, near every point on the divisor, the Poincar\'e type metric is asymptotic to the model product metric given by the Poincar\'e cusp metric on the punctured disc and a smooth metric on the divisor. Known constructions of canonical K\"ahler metrics of Poincar\'e type include the negative K\"ahler-Einstein ones studied in \cite{CY80},  \cite{Kob84}, \cite{TY87}; the toric ones studied in \cite{Abr01} and \cite{Bry01}; and the constant scalar curvature K\"ahler and extremal K\"ahler ones studied in \cite{Sek18} and \cite{Fen24}. Most of these metrics are \textit{not} explicit, while in this article the constructions of toric scalar-flat K\"ahler metrics of Poincar\'e type are explicit.

\medskip

First, we focus on the case where $X$ is strictly unbounded, or equivalently, as a complex surface, it arises as a finite sequence of blow-ups of minimal resolution of $\C^2\slash\Gamma$, as described in Definition-Proposition~\ref{definition strictly unbounded}:

\begin{thm}\label{main theorem}(Theorem~\ref{Theorem unbounded PTK SFK}+Theorem~\ref{theorem asymptotic behavior})
   Given $X$ a strictly unbounded toric symplectic $4$-manifold and $D=\displaystyle\sum_{i=1}^mD_i$ a divisor on $X$ such that each $D_i$ is an irreducible prime divisor fixed by the torus action and $D_i\cap D_j=\emptyset$ for any $1\le i\ne j\le m$. On $X\backslash D$, we have

   \begin{itemize}
       \item  a toric scalar-flat K\"ahler metric with Poincar\'e type singularity along $D$, and is asymptotically locally Euclidean(ALE) on the remaining end;
       \item  a two-parameter family of toric scalar-flat K\"ahler metrics with Poincar\'e type singularity along $D$, and are asymptotic to either the generalized Taub-NUT metrics or the exceptional Taub-NUT metrics on the remaining end.
   \end{itemize}
\end{thm}

Along the proof, we will see the complex structures induced from these metrics are biholomorphic to that on $X$ away from the divisor $D$. Here the \textit{generalized Taub-NUT metrics} live on $\C^2$ and are scalar-flat K\"ahler generalizations of the Ricci-flat K\"ahler Taub-NUT metrics. These metrics are introduced by Donaldson in \cite{donaldson2009constant} and later explored by Abreu and Sena-Dias in \cite{AS12} and Weber in \cite{weber2022asymptotic} and \cite{weber2022generalized}. They all have quadratic curvature decay and cubic volume growth, but except for the standard Taub-NUT metric, they are not asymptotically locally flat(ALF). The \textit{exceptional Taub-NUT metrics} are also scalar-flat K\"ahler metrics living on $\C^2$. These metrics are studied by Weber in \cite{weber2022generalized}, where he showed these metrics have quadratic curvature decay and quartic volume growth but are not ALE.

\medskip

A previously known example of toric scalar-flat K\"ahler metric of Poincar\'e type is first discussed by Fu-Yau-Zhou in \cite{FYZ16}, and later studied by the author in \cite{Fen24}, which we refer to as the \textit{Hwang-Singer metric} $\omega_{HS}$. It lives on $\C^2-\{0\}$. Near the origin, it has Poincar\'e type singularity, and it is asymptotically Euclidean on the other end. This $S^1$-invariant metric is in fact toric and we will discuss it in detail in Example~\ref{HS}. Theorem~\ref{main theorem} then gives us a two-parameter family deformation of $\omega_{HS}$. 

\medskip

Intuitively, Poincar\'e type metrics can be viewed as the limit of a
 conical family of metrics when the cone angle approaches $0$. Similarly, smooth metrics can be viewed as the limit when the cone angle approaches $2\pi$. This was proved by Guenancia \cite{Gue15} in the K\"ahler-Einstein setting. For the metrics constructed in Theorem~\ref{main theorem}, we can explicitly write down a conical family of toric K\"ahler metrics connecting the Poincar\'e type metrics with those in \cite{AS12}. This family of conical metrics is \textit{not} necessarily scalar-flat, though. 
 
 \medskip
 
 On the other hand, Weber \cite{weber2023analytic} gave a construction of toric scalar-flat K\"ahler metrics with conical singularity along the divisor. The cone angle along a given edge is closely related to the notion of "label" introduced there. In \cite{weber2023analytic}, the label is interpreted as a characterization of the growth speed of the Killing field vanishing along the edge. We formulate the problem from a different perspective, emphasizing the various \textit{boundary conditions} specified by the cone angles. More precisely, following the method of Abreu and Sena-Dias, we give an independent construction of the conical toric scalar-flat K\"ahler metrics:

\begin{thm}(Theorem~\ref{Theorem unbounded conical SFK})
   Consider the same setting as in Theorem~\ref{main theorem}. Fix $\theta_i\in(0, 1)$ for $i=1, \cdots, m$, on $X$, we have

   \begin{itemize}
       \item  a conical toric scalar-flat K\"ahler metric with angle $2\pi\theta_i$ along $D_i$ and is asymptotically locally Euclidean(ALE) on the remaining end;
       \item  a two-parameter family of them  with angle $2\pi\theta_i$ along $D_i$ and are asymptotic to either the generalized Taub-NUT metrics or the exceptional Taub-NUT metrics on the remaining end.
   \end{itemize}

\end{thm}

Besides the general case where $X$ is strictly unbounded, we construct toric scalar-flat K\"ahler metrics of Poincar\'e type when the unbounded edges of the momentum polytope are parallel:

\begin{thm}\label{theorem parallel edges}
    Consider the same setting as in Theorem~\ref{main theorem} except that the unbounded edges of $X$ are parallel. On $X\backslash D$, we have a one-parameter family of toric scalar-flat K\"ahler metrics with Poincar\'e type singularity along $D$ and are asymptotic to the model product metric on $S^2\times\R^2$ on the remaining end.
\end{thm}

The scalar-flat metrics we constructed belong to a particular class of Poincar\'e type metrics, characterized by the specific behavior of their potential functions along the divisor, which we denote as the $S_{\al, \beta}$ type (Definition~\ref{definition S type}). We have the following \textit{uniqueness} result:

\begin{thm}(Theorem~\ref{theorem uniqueness at the end})\label{theorem uniqueness}
    Given the same setting as in Theorem~\ref{main theorem}. Assume $g$ is a toric scalar-flat K\"ahler metric on $X\backslash D$, and its symplectic potential $u$ is of $S_{\al, \beta}$ type along $D$, then $g$ can only be one of the metrics constructed in Theorem~\ref{main theorem}.
\end{thm}

Naturally, we would ask if we still have the uniqueness result without assuming $u$ to be of $S_{\al, \beta}$ type: 

\begin{?}(Strong uniqueness)
    Without assuming the potential function is of $S_{\al, \beta}$ type along the divisor in Theorem~\ref{theorem uniqueness}, can we still obtain the uniqueness result? 
\end{?}

A related question is to determine, locally, whether $S_{\al, \beta}$ type represents the only Guillemin boundary behavior for scalar-flat K\"ahler metric of Poincar\'e type.

\paragraph{Outline of the article.} In Section~\ref{Section prelim}, we discuss the preliminaries of the construction. In Section~\ref{Section construction of PT}, under a specific boundary condition, we give an explicit construction of toric scalar-flat K\"ahler metrics of Poincar\'e type using Donaldson's local ansatz for scalar-flat K\"ahler metrics, and discuss their asymptotic behavior, hence proving Theorem~\ref{main theorem} and Theorem~\ref{theorem parallel edges}. We also include a discussion on the example of the Hwang-Singer metric. In Section~\ref{Section construction of conic}, we use similar arguments to construct a family of conical toric scalar-flat K\"ahler metrics. In the Appendix \ref{Section}, we show the uniqueness result with prescribed explicit boundary behavior along the divisor.

\paragraph{Acknowledgements} The author is deeply grateful to her advisor, Song Sun, for inspiring discussions and patient guidance. She also sincerely thanks Vestislav Apostolov, Charles Cifarelli, and Lars Martin Sektnan for valuable suggestions on an earlier version of this work, and Rosa Sena-Dias for clarifying a question addressed in \cite{AS12}. The work was partially supported by NSF Grant DMS-2304692 and conducted in part during the author's stay at IASM, whose support is appreciated.

\section{Preliminaries}\label{Section prelim}

In this section, we recall some basics of toric K\"ahler metrics on toric symplectic $4$-manifolds and discuss the local ansatz for finding scalar-flat K\"ahler metrics. First, we recall the basic definition of a toric symplectic $4$-manifold:

\begin{definition}\label{toric manifold}
    A symplectic $4$-manifold $(X, \omega)$ is said to be \textbf{toric} if it admits an effective Hamiltonian $\mathbb{T}^2$-action $\tau$ of the standard torus to the diffeomorphism group of $(X, \omega)$ such that the corresponding moment map $\mu: X\rightarrow\R^2$ is proper onto its image $P$.
\end{definition}

Here the moment map of the $\mathbb{T}^2$-action is a map $\mu: X\rightarrow \R^2$ such that $\iota_{\xi}\omega=-d\mu$ for each infinitesimal generator $\xi$ of $\mathbb{T}^2$. For a compact symplectic $4$-manifold, the moment image of $\mu$ is the convex hull of the image of fixed points of $\mathbb{T}^2$ in $X$, and the classical Atiyah-Guillemin-Sternberg tells us this image is a polytope. For a non-compact symplectic $4$-manifold, we first introduce the definition of the moment polytope:

\begin{definition}(\cite{AS12}, Definition 2.2)
    We say a convex polytope $P\subset\R^2$ is a moment polytope if

    \begin{enumerate}[(i)]
    \item for each edge, we can find a primitive vector of $\Z^2$, which is an interior normal to this edge;
    \item for each pair of intersecting edges, their chosen interior normals form a $\Z$-basis of $\Z^2$. 
    \end{enumerate}
\end{definition}

We say two moment polytopes are equivalent if there exists a translation in $\R^2$ and a $GL(2, \Z)$ transformation mapping one to the other, and two symplectic toric manifolds are equivalent if there is an equivariant symplectomorphism mapping one to the other. Delzant's theorem \cite{delzant1988hamiltoniens} tells us in the compact setting, the moment polytope determines the symplectic toric $4$-manifold up to equivariant symplectomorphism. It turns out that in the non-compact setting, we also have the following correspondence:

\begin{prop}(\cite{KL15})\label{correspondene for noncompact}
    There is a bijective correspondence between the equivalence class of symplectic toric $4$-manifolds and the equivalence class of moment polytopes. 
\end{prop}

We are particularly interested in the following class of non-compact symplectic toric $4$-manifold:

\begin{def-prop}(\cite{AS12} Definition 2.4, Proposition 2.8)\label{definition strictly unbounded})
A symplectic toric $4$-manifold is said to be \textbf{strictly unbounded} if the following equivalent conditions hold:

\begin{enumerate}[(i)]
    \item the image of the moment map, $P$, is an unbounded polytope with finitely many edges, with the unbounded edges being non-parallel;
    \item $X$ as a complex surface, arises as a finite sequence of blow-ups of a minimal resolution of $\C^2\slash\Gamma$ for some finite cyclic group $\Gamma\subset U(2)$ such that $\C^2\slash\Gamma$ has an isolated singularity point at the origin.
\end{enumerate}
\end{def-prop}

Let $P$ be a moment polytope given by 

\begin{equation}\label{definition of moment polutope}
P\coloneqq\{x\in \R^2: \ell_i(x)\coloneqq\langle x, \nu_i\rangle+\lambda_i\ge 0, i=1, \cdots, d\}.
\end{equation}

Here $\nu_i=(\alpha_i, \beta_i)\in\mathbb{Z}^2$ are the primitive interior normals to the edges. We order its edges so that $\ell_1$, $\ell_d$ are the unbounded edges and $\ell_i\cap\ell_{i+1}\ne\emptyset$ for $i=1, \cdots, d-1$. From \cite{AS12} Remark~2.5, we know we can further assume the Delzant condition for the polytope:

$$\det(\nu_i, \nu_{i+1})=-1, \text{ for }i=1, \cdots, d-1.$$

Let $(X_P, \omega_P, \tau_P)$ be the associated symplectic $4$-manifold of $P$ with moment map $\mu_P$. It admits a canonical integrable, torus-invariant, compatible complex structure $J_P$. We denote the resulting K\"ahler surface by $(X_P, \omega_P, J_P, g_P)$. Let $P^{\circ}$ be the interior of $P$, and consider $X^{\circ}_P\coloneqq\mu_P^{-1}(P^{\circ})$, then 

$$X^{\circ}_P\cong P^{\circ}\times \mathbb{T}^2=\{(x, \theta): x=(x_1, x_2)\in P^{\circ}, \theta=(\theta_1, \theta_2)\in\R^2\slash\Z^2\}.$$

Here $(x, \theta)$ are interpreted as the \textit{\textbf{action-angle coordinates}} for $\omega_P$, i.e.,

$$\omega_P=dx_1\wedge d\theta_1+dx_2\wedge d\theta_2.$$

From \cite{Gui94}, we know the symplectic potential $u_P\in C^{\infty}(P^{\circ})$ is written as

$$u_P=\displaystyle\frac{1}{2}\displaystyle\sum_{i=1}^d\ell_i(x)\log\ell_i(x).$$

The metric $g_P$ is given by

$$g_P=\displaystyle\sum_{i, j=1}^2((\Hess u_P)_{ij}dx_i\otimes dx_j+(\Hess u_P)^{ij}d\theta_i\otimes d\theta_j).$$

From \cite{Abr98} and \cite{apostolov2004hamiltonian}, given any toric complex structure $J$ which is $\omega_P$-compatible, there exist action-angle coordinates $(x, \theta)$ on $P^{\circ}$ such that for some symmetric and positive-definite $2\times 2$ matrix $U(x)$, we can write $J$ in the following form

$$J=-\displaystyle\sum_{i, j=1}^2(U(x)^{ij}\displaystyle\frac{\partial}{\partial x^i}\otimes d\theta_j+U(x)_{ij}\displaystyle\frac{\partial}{\partial \theta^i}\otimes dx_j).$$

Furthermore, the integrability of $J$ is equivalent to the existence of $u\in C^{\infty}(P^{\circ})$ such that $U(x)=\Hess_x(u)$. Then $u$ is the potential corresponding to $J$, and the K\"ahler metric is written as

\begin{equation}\label{Kahler metric potential}
    g=\displaystyle\sum_{i, j=1}^2((\Hess u)_{ij}dx_i\otimes dx_j+(\Hess u)^{ij}dx_i\otimes dx_j).
\end{equation}

For simplicity concern, we will use $u_{ij}, u^{ij}$ to denote $(\Hess u)_{ij}, (\Hess u)^{ij}$ respectively. From \cite{Abr03}, \cite{Abr01}, we know when the Hessian of the symplectic potential $u$ on $P^{\circ}$ is positive-definite and the boundary behavior of $u$ is specified by the Guillemin's boundary condition, it determines a complex structure on $X^{\circ}_P$ which extends to $(X_P, \omega_P, \tau_P)$. We say $u$ satisfies Guillemin's boundary condition if modulo a smooth function,

   \begin{equation}\label{Guillemin boundary condition}
	u(x)=\displaystyle\frac{1}{2}\displaystyle\sum_{i=1}^d\ell_i(x)\log\ell_i(x),
   \end{equation}
   
and its restriction to the interior of each face of $P$ is strictly convex and smooth. 

\begin{definition}(\cite{AAS17}, Definition 4.2)
	Given $P$, write $L\coloneqq\{\ell_1(x), \cdots, \ell_d(x)\}$. We say a symplectic potential $u: P^{\circ}\rightarrow\R$ belongs to the class $S(P, L)$ if it is smooth, strictly convex, and satisfies the Guillemin boundary condition.
\end{definition}

We mention that as discussed in \cite{apostolov2004hamiltonian} and \cite{AAS17} Proposition~4.3, there is an equivalent characterization of $S(P, L)$, which we refer to as the first-order boundary conditions. 

\medskip

This article focuses on finding scalar-flat K\"ahler metrics. Direct calculations show the scalar curvature of the metric has the following expression:

$$s=-\displaystyle\sum_{i, j}\displaystyle\frac{\partial^2u^{ij}}{\partial x_i\partial x_j}.$$

Then the scalar-flat equation we aim at solving becomes $\displaystyle\sum_{i, j}\displaystyle\frac{\partial^2u^{ij}}{\partial x_i\partial x_j}=0.$ In \cite{Don09}, Donaldson gave a reformulation of Joyce's construction in \cite{Joy95}, which allows us to write down explicit symplectic potentials of scalar-flat K\"ahler metrics on complex surfaces. The key is to use the \textbf{\textit{axi-symmetric harmonic function}} as the local model. More precisely:

\begin{thm}(\cite{Don09}, local model)\label{theorem local model}
	Let $\xi_1, \xi_2$ be two solutions to 
	
	\begin{equation}\label{axi-symmetric harmonic function}
	\displaystyle\frac{\partial^2\xi}{\partial H^2}+\displaystyle\frac{\partial^2\xi}{\partial r^2}+\displaystyle\frac{1}{r}\displaystyle\frac{\partial\xi}{\partial r}=0
    \end{equation}
    
    on 
    
    $$\h\coloneqq\{(H, r)\in\R^2: r>0\}.$$
    
    Then the $1$-forms
    
    $$\epsilon_1=r\left(\displaystyle\frac{\partial\xi_2}{\partial r}dH-\displaystyle\frac{\partial\xi_2}{\partial H}dr\right), \quad \epsilon_2=-r\left(\displaystyle\frac{\partial\xi_1}{\partial r}dH-\displaystyle\frac{\partial\xi_1}{\partial H}dr\right)$$ 
    
    are closed. Let $x_1, x_2$ be their primitives, then the $1$-form $\ep=\xi_1dx_1+\xi_2dx_2$ is also closed. Let $u$ be its primitive. Assume for $\xi=(\xi_1, \xi_2)$, we have 
    
    \begin{equation}\label{condition local model}
    	\det D\xi>0.
    \end{equation}
    
    Then $u$ is a local symplectic potential for a scalar-flat K\"ahler toric metric on $\R^4$.    
\end{thm}

Some known solutions to (\ref{axi-symmetric harmonic function}) include

\begin{equation}\label{solutions to axi-symmetric}
    aH+b, \quad a\log r+b, \quad \displaystyle\frac{1}{2}\log\left(H+\displaystyle\sqrt{(H+a)^2+r^2}\right)
\end{equation}

where $a, b\in\R$. In \cite{AS12}, the authors used these solutions to construct scalar-flat K\"ahler metrics on unbounded symplectic toric $4$-manifolds. These metrics belong to the class $S(P, L)$, and are precisely those whose complex structures are equivariantly biholomorphic to $J_P$. 

\medskip

The situation is different for toric scalar-flat K\"ahler metrics of Poincar\'e type. We first recall the definition of Poincar\'e type K\"ahler metrics:

\begin{definition} (Poincar\'e type K\"ahler metric, \cite{Auv17})\label{definition PT}
    Given $(X, \omega_0)$ a compact complex manifold and $D$ a smooth divisor in $X$ with $\sigma\in H^0(X, \mathcal{O}(D))$ being a holomorphic defining section. Fix $\lambda$ a sufficiently large constant such that
    
    $$\omega_h\coloneqq \omega_0-\displaystyle\sqrt{-1}\partial\bar{\partial}\log(\lambda-\log(|\sigma|^2))$$

    is a positive $(1, 1)$-form on $X\backslash D$. We say a closed, smooth $(1,1)$-form  
    $$\omega_{PT}\coloneqq \omega_0+\sqrt{-1}\partial\bar{\partial}\varphi$$
    
    on $X\backslash D$ is a \textbf{Poincar\'e type K\"ahler metric} if 
    \begin{itemize}
        \item $\omega_{PT}$ is \textit{quasi-isometric} to $\omega_{h}$, which means there exists some $C>0$, such that $\frac{1}{C}\omega_h\le\omega_{PT}\le C\omega_{h}$ and $\forall i\ge 1$, $\sup_{X\backslash D}|\nabla^i_{\omega_h}\omega_{PT}|<\infty$;
        \item $\varphi$ is a smooth function on $X\backslash D$ with $\varphi=O(h)$, and $\forall i\ge 1$, $\sup_{X\backslash D}|\nabla^i_{\omega_{h}}\varphi|<\infty$.
    \end{itemize}
\end{definition}

Let $D$ be a torus-invariant divisor in $X$, and let $\ell_F$ be the edge corresponding to $D$ in the moment polytope. In \cite{AAS17}, Section 4.3, the authors introduced a special type of Guillemin boundary condition for symplectic potential $u$, which gives rise to a Poincar\'e type K\"ahler metric. For the setting of complex surfaces, we recall the definition as follows:

\begin{definition}(\cite{AAS17}, Definition 4.16)\label{definition S type}
	Given $\al\in\R^+, \beta\in \R$, we say a symplectic potential $u: P^{\circ}\rightarrow\R$ belongs to the class $S_{\al, \beta}(P, L, F)$ if it is strictly convex and smooth on $P^{\circ}$, its restriction to the interior of each edge of $P$ is strictly convex and smooth, and
		
	\begin{equation}
		u+(\al+\beta\ell_F)\log\ell_F-\displaystyle\frac{1}{2}\displaystyle\sum_{j=2}^d\ell_j\log\ell_j
	\end{equation}
	
	is smooth on $P$.
\end{definition} 

From \cite{AAS17} Theorem 4.18, we know for $u\in S_{\al, \beta}(P, L, F)$, the induced K\"ahler metric (\ref{Kahler metric potential}) exhibits Poincar\'e type behavior along $D$. Let $d\lambda$ be the Lebesgue measure on $\R^2$, we define a measure $d\lambda_i$ on $\ell_i$ by 

$$-d\ell_i\wedge d\lambda_i=-\nu_i\wedge d\lambda_i=d\lambda.$$

On $\ell_F$, the induced measure is zero for the class $S(P, L, F)$, which is obtained by sending the corresponding label $\ell_i$ to infinity. Although the Guillemin boundary condition does not have a straightforward extension to describe the behavior of $u$ near $\ell_F$, the first-order boundary condition does, as pointed out in \cite{AAS17} Definition~4.6. Note $S_{\al, \beta}(P, L, F)$ is only a proper subset of $S(P, L, F)$, which we see by comparing \cite{AAS17} Definition~4.6 and Proposition~4.19. 

\medskip

We write 

$$u_{P, F, \al, \beta}=\displaystyle\frac{1}{2}\displaystyle\sum_{j=2}^d\ell_j\log\ell_j-(\al+\beta\ell_F)\log\ell_F$$

as the potential of the model metric for the class $S_{\al, \beta}(P, L, F)$. To construct scalar-flat K\"ahler metrics of Poincar\'e type in this class, we need other solutions to (\ref{axi-symmetric harmonic function}) besides the ones (\ref{solutions to axi-symmetric}). We consider

$$\xi=\displaystyle\frac{1}{2}\displaystyle\frac{1}{\displaystyle\sqrt{(H+a)^2+r^2}}, \quad a\in\R.$$

This solution, together with the solutions to (\ref{axi-symmetric harmonic function}) mentioned above, serve as local models for our construction in the next section.

\section{Construction of scalar-flat K\"ahler metrics of Poincar\'e type}\label{Section construction of PT}

Given $X$ a strictly unbounded symplectic toric $4$-manifold and let $P$ be its moment polytope defined by (\ref{definition of moment polutope}). Write $L=\{\ell_1(x), \cdots, \ell_d(x)\}$. Let $D=\displaystyle\sum_{i=1}^m D_i$ be a smooth divisor on $X$ such that each $D_j$ is fixed by the torus action. Let $\ell_{i_j}$ be its image on the moment polytope. Assume 

\begin{equation}\label{non-adjacent index}
  i_1>1, i_2>i_1+1, \cdots, d>i_m>i_{m-1}+1.
\end{equation}

Let $I=\{i_1, \cdots i_m\}\subset\{1, 2, \cdots, d\}$ be the index set and    $\ell_{I}=\bigcup_{j=1}^m\ell_{i_j}$ be the union of the edges corresponding to $D_i$, then $P\backslash \ell_I$ is the moment polytope of $X\backslash D$. 

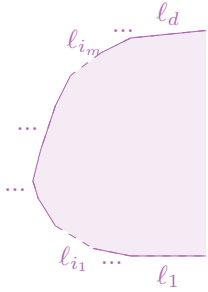
\begin{figure}[h]
\begin{center}
    \begin{tikzpicture}
       \filldraw[color=violet!1, fill=violet!8, ultra thin] (2,3) --(1, 2.9)-- (0.8,2.8) -- (0.6,2.7) --(0.2, 2.4) -- (0, 2) --(-0.2, 1.4)-- (-0.3, 1) --(-0,1, 0.6)-- (0, 0.4) --(0.5, 0.1)--(1, 0) -- (2, 0);
       \filldraw[color=violet!60, fill=violet!8] (2,3) --(1, 2.9)-- (0.8,2.8) -- (0.6,2.7) (0.2, 2.4) -- (0, 2) --(-0.2, 1.4)-- (-0.3, 1) --(-0,1, 0.6)-- (0, 0.4) (0.5, 0.1)--(1, 0) -- (2, 0);
       \draw[color=violet!60, dashed](0, 0.4)--node[below] {$\ell_{i_1}$} (0.5, 0.1);
       \draw[color=violet!60, dashed](0.6,2.7)--node[above] {$\ell_{i_m}$} (0.2, 2.4);
       \draw[color=violet!60, dashed](0, 2)--node[left] {$...$}(-0.2, 1.4); 
       \draw[color=violet!60, dashed](-0.3, 1)--node[left] {$...$}(-0,1, 0.6); 
       \draw[color=violet!60, dashed](1, 2.9)--node[above] {$...$}(0.8,2.8); 
       \draw[color=violet!60, dashed](2,3)--node[above] {$\ell_d$}(1, 2.9);
        \draw[color=violet!60, dashed](0.5, 0.1)--node[below] {$...$}(1, 0); 
       \draw[color=violet!60, dashed](1, 0)--node[below] {$\ell_1$}(2, 0);
    \end{tikzpicture}
    \caption{The moment polytope $P\backslash\ell_I$} 
\end{center}    
 \end{figure} 

We prove Theorem~\ref{main theorem} by giving an explicit construction of the toric scalar-flat K\"ahler metrics of Poincar\'e type on $X\backslash D$. Consider $\nu=(\alpha, \beta)$ a vector in $\mathbb{R}^2$ s.t. 

\begin{equation}\label{conditions on nu}
    \det(\nu, \nu_1), \det(\nu, \nu_d)\ge 0.
\end{equation}

As discussed in \cite{AS12}, since $P$ is strictly unbounded, this set of vectors forms a cone bounded by $-\nu_1$ and $\nu_d$.

\begin{thm} \label{Theorem unbounded PTK SFK}

For $X, D, P, I, \nu$ defined as above, there exist constants $\Lambda_{i_1}, \cdots, \Lambda_{i_m}>0$ determined by the polytope $P$, for $(H, r)\in \h$, set

$$\xi_1\coloneqq\alpha_1\log r+\displaystyle\frac{1}{2}\displaystyle\sum_{i=1}^{d-1}(\alpha_{i+1}-\alpha_i)\log\left(H+a_i+\displaystyle\sqrt{(H+a_i)^2+r^2}\right)-\displaystyle\frac{1}{2}\displaystyle\sum_{k\in I}\displaystyle\frac{\Lambda_k\alpha_k}{\displaystyle\sqrt{(H+a_{k-1})^2+r^2}}+\alpha H,$$

$$\xi_2\coloneqq\beta_1\log r+\displaystyle\frac{1}{2}\displaystyle\sum_{i=1}^{d-1}(\beta_{i+1}-\beta_i)\log\left(H+a_i+\displaystyle\sqrt{(H+a_i)^2+r^2}\right)-\displaystyle\frac{1}{2}\displaystyle\sum_{k\in I}\displaystyle\frac{\Lambda_k\beta_k}{\displaystyle\sqrt{(H+a_{k-1})^2+r^2}}+\beta H.$$

Here $a_1, \cdots, a_{d-1}$ are real numbers determined by $P$ satisfying 

\begin{equation}\label{conditions on polytope reals}
    a_{j-1}>a_j \text{ if } j\notin I \text{ and } a_{j-1}=a_j \text{ if } j\in I. 
\end{equation}

Let $x_1, x_2$ be the primitives of 

$$\epsilon_1=r\left(\displaystyle\frac{\partial\xi_2}{\partial r}dH-\displaystyle\frac{\partial\xi_2}{\partial H}dr\right), \quad \epsilon_2=-r\left(\displaystyle\frac{\partial\xi_1}{\partial r}dH-\displaystyle\frac{\partial\xi_1}{\partial H}dr\right).$$

Then they define the momentum action coordinates on $P^{\circ}$ of some toric scalar-flat K\"ahler metric of Poincar\'e type on $X\backslash D$ whose symplectic potential satisfies

$$du=\xi_1dx_1+\xi_2dx_2.$$

Furthermore, for each $k\in I$,

\begin{equation}\label{target boundary behavior}
u\in S_{\frac{\Lambda_k}{2}, \frac{1}{2}\det(\nu_{k-1}, \nu_{k+1})}(P, L, \ell_k).
\end{equation}

\end{thm}

\begin{proof}

The first step is to show the assumption (\ref{condition local model}) in Theorem~\ref{theorem local model} is satisfied. We compute $D\xi$:

\begin{equation}\label{determinant}
    D\xi=\begin{pmatrix}
    \alpha+\displaystyle\frac{1}{2}\displaystyle\sum_{i=1}^{d-1}\displaystyle\frac{\alpha_{i+1}-\alpha_i}{\rho_i}+\displaystyle\frac{1}{2}\displaystyle\sum_{k\in I}\displaystyle\frac{\Lambda_k\alpha_kH_{k-1}}{\rho_{k-1}^3}&\displaystyle\frac{\alpha_1}{r}+\displaystyle\frac{1}{2}\displaystyle\sum_{i=1}^{d-1}\displaystyle\frac{(\alpha_{i+1}-\alpha_i)r}{\rho_i(\rho_i+H_i)}+\displaystyle\frac{1}{2}\displaystyle\sum_{k\in I}\displaystyle\frac{\Lambda_k\alpha_kr}{\rho_{k-1}^3}\\
    \beta+\displaystyle\frac{1}{2}\displaystyle\sum_{i=1}^{d-1}\displaystyle\frac{\beta_{i+1}-\beta_i}{\rho_i}+\displaystyle\frac{1}{2}\displaystyle\sum_{k\in I}\displaystyle\frac{\Lambda_k\beta_{k}H_{k-1}}{\rho_{k-1}^3}&\displaystyle\frac{\beta_1}{r}+\displaystyle\frac{1}{2}\displaystyle\sum_{i=1}^{d-1}\displaystyle\frac{(\beta_{i+1}-\beta_i)r}{\rho_i(\rho_i+H_i)}+\displaystyle\frac{1}{2}\displaystyle\sum_{k\in I}\displaystyle\frac{\Lambda_k\beta_kr}{\rho_{k-1}^3}
\end{pmatrix}.
\end{equation}

Here $H_i\coloneqq H+a_i$, $\rho_i\coloneqq\displaystyle\sqrt{H_i^2+r^2}$. Set $a_0\coloneqq-\infty$, $a_d\coloneqq\infty$ and set $H_0, H_d, \rho_0, \rho_d$ accordingly, we rewrite Equation~(\ref{determinant}) as

\begin{equation}\label{determinant rewritten}
    D\xi=\begin{pmatrix}
    \alpha+\displaystyle\frac{1}{2}\displaystyle\sum_{i=1}^{d}\alpha_i\left(\displaystyle\frac{1}{\rho_{i-1}}-\displaystyle\frac{1}{\rho_{i}}\right)+\displaystyle\frac{1}{2}\displaystyle\sum_{k\in I}\displaystyle\frac{\Lambda_k\alpha_{k}H_{k-1}}{\rho_{k-1}^3}&\displaystyle\frac{1}{2r}\displaystyle\sum_{i=1}^{d}\alpha_i\left(\displaystyle\frac{H_{i-1}}{\rho_{i-1}}-\displaystyle\frac{H_{i}}{\rho_{i}}\right)+\displaystyle\frac{1}{2}\displaystyle\sum_{k\in I}\displaystyle\frac{\Lambda_k\alpha_kr}{\rho_{k-1}^3}\\
    \beta+\displaystyle\frac{1}{2}\displaystyle\sum_{i=1}^{d}\beta_i\left(\displaystyle\frac{1}{\rho_{i-1}}-\displaystyle\frac{1}{\rho_{i}}\right)+\displaystyle\frac{1}{2}\displaystyle\sum_{k\in I}\displaystyle\frac{\Lambda_k\beta_{k}H_{k-1}}{\rho_{k-1}^3}&\displaystyle\frac{1}{2r}\displaystyle\sum_{i=1}^{d}\beta_i\left(\displaystyle\frac{H_{i-1}}{\rho_{i-1}}-\displaystyle\frac{H_{i}}{\rho_{i}}\right)+\displaystyle\frac{1}{2}\displaystyle\sum_{k\in I}\displaystyle\frac{\Lambda_k\beta_kr}{\rho_{k-1}^3}
\end{pmatrix}.
\end{equation}

Note from \cite{AS12} Theorem 4.1, we know the determinant of the following matrix is positive:

\begin{equation}\label{AS's matrix}
    \begin{pmatrix}
    \alpha+\displaystyle\frac{1}{2}\displaystyle\sum_{i=1}^{d}\alpha_i\left(\displaystyle\frac{1}{\rho_{i-1}}-\displaystyle\frac{1}{\rho_{i}}\right)&\displaystyle\frac{1}{2r}\displaystyle\sum_{i=1}^{d}\alpha_i\left(\displaystyle\frac{H_{i-1}}{\rho_{i-1}}-\displaystyle\frac{H_{i}}{\rho_{i}}\right)\\
    \beta+\displaystyle\frac{1}{2}\displaystyle\sum_{i=1}^{d}\beta_i\left(\displaystyle\frac{1}{\rho_{i-1}}-\displaystyle\frac{1}{\rho_{i}}\right)&\displaystyle\frac{1}{2r}\displaystyle\sum_{i=1}^{d}\beta_i\left(\displaystyle\frac{H_{i-1}}{\rho_{i-1}}-\displaystyle\frac{H_{i}}{\rho_{i}}\right)
\end{pmatrix}.
\end{equation}

Comparing the expression of $\det(D\xi)$ with the above, it suffices to show their difference is still positive. A key step in showing the positivity is the following lemma:

\begin{lem}
    Given $k\in I$, $\forall i$, the term involving $i$ and $k$ in the expression of $\det(D\xi)$ has the following expression and is non-negative:

\begin{equation}\label{component for cusp edge with normal edge}
    (\beta_k\alpha_i-\alpha_k\beta_i)\displaystyle\frac{1}{4r\rho_{k-1}^3}\left(\displaystyle\frac{r^2+H_{k-1}H_{i-1}}{\rho_{i-1}}-\displaystyle\frac{r^2+H_{k-1}H_i}{\rho_{i}}\right).
\end{equation}

\end{lem}

\begin{proof}

    We rewrite Equation~(\ref{component for cusp edge with normal edge}) as follows:
    
    $$(\beta_k\alpha_i-\alpha_k\beta_i)\displaystyle\frac{1}{4\rho_{k-1}^3}\left(r\left(\frac{1}{\rho_{i-1}}-\frac{1}{\rho_i}\right)+H_{k-1}\left(\frac{H_{i-1}}{r\rho_{i-1}}-\frac{H_i}{r\rho_i}\right)\right).$$
    
   Note for any $i$, $r, \rho_i>0$, and from $\det(\nu_i, \nu_{i+1})=-1$ we deduce that $\beta_k\alpha_i-\alpha_k\beta_i>0 \iff i>k$. Thus, it suffices to show

    \begin{equation}\label{claim i>k}
        \rho_i(r^2+H_{k-1}H_{i-1})-\rho_{i-1}(r^2+H_{k-1}H_i)>0 \text{ for }i>k
    \end{equation}

    and 

    \begin{equation}
        \rho_i(r^2+H_{k-1}H_{i-1})-\rho_{i-1}(r^2+H_{k-1}H_i)<0 \text{ for }i<k.
    \end{equation}

    When $i>k$, we rewrite the expression as

    \begin{equation}
        \rho_i(r^2+H_{k-1}H_{i-1})-\rho_{i-1}(r^2+H_{k-1}H_i)=\displaystyle\frac{(\rho_i-\rho_{i-1})r^2}{\rho_iH_{i-1}+\rho_{i-1}H_i}(\rho_i(H_{i-1}-H_{k-1})+\rho_{i-1}(H_i-H_{k-1})).
    \end{equation}

    We claim that $\rho_{i-1}H_i-\rho_iH_{i-1}>0$. It is because $f(x)=\displaystyle\frac{x}{\displaystyle\sqrt{x^2+y^2}}$ is an increasing function given $y>0$, then $\displaystyle\frac{H_i}{\displaystyle\sqrt{H_i^2+r^2}}>\displaystyle\frac{H_{i-1}}{\displaystyle\sqrt{H_{i-1}^2+r^2}}.$ Hence, we know

    \begin{equation}\label{fact01}
        \rho_iH_{i-1}+\rho_{i-1}H_i>0\iff (\rho_{i-1}H_i-\rho_iH_{i-1})(\rho_iH_{i-1}+\rho_{i-1}H_i)>0\iff r^2(H_i^2-H_{i-1}^2)>0.
    \end{equation}
    
    On the other hand, we have

    \begin{equation}\label{fact02}
        \rho_i-\rho_{i-1}>0\iff (\rho_i+\rho_{i-1})(\rho_i-\rho_{i-1})>0\iff H_i^2-H_{i-1}^2>0.
    \end{equation}

    Combining $H_{i-1}-H_{k-1}>0, H_i-H_{k-1}>0$ with (\ref{fact01}) and (\ref{fact02}), we obtain (\ref{claim i>k}). Similarly, for $i<k$, 

    \begin{equation}
        \rho_i(r^2+H_{k-1}H_{i-1})-\rho_{i-1}(r^2+H_{k-1}H_i)=\displaystyle\frac{(\rho_i-\rho_{i-1})r^2}{\rho_iH_{i-1}+\rho_{i-1}H_i}(\rho_i(H_{i-1}-H_{k-1})+\rho_{i-1}(H_i-H_{k-1}))<0.
    \end{equation}    
\end{proof}

For $k, k'\in I$, the term involving $k$ and $k'$ in $\displaystyle\sum_{k\in I}\displaystyle\frac{\al_kH_{k-1}}{\rho_{k-1}^3}\displaystyle\sum_{k\in I}\displaystyle\frac{\beta_kr}{\rho_{k-1}^3}-\displaystyle\sum_{k\in I}\displaystyle\frac{\al_kr}{\rho_{k-1}^3}\displaystyle\sum_{k\in I}\displaystyle\frac{\beta_{k}H_{k-1}}{\rho_{k-1}^3}$ has the following expression and is non-negative:

\begin{equation}\label{component within cusp edges}
    \displaystyle\frac{r}{\rho_k^3\rho_{k'}^3}(\alpha_k\beta_{k'}-\alpha_{k'}\beta_k)(H_{k-1}-H_{k'-1}).
\end{equation}

Finally, for $\nu$ satisfying (\ref{conditions on nu}), given $k\in I$, $\det D\xi$ changes by adding 

\begin{multline}\label{component for nu}
    \det(\nu, \nu_1)\left(\displaystyle\frac{1}{r}-\displaystyle\frac{r}{2\rho_1(H_1+\rho_1)}\right)+\displaystyle\frac{r}{2}\displaystyle\sum_{i=1}^{d-1}\det(\nu, \nu_i)\left(\displaystyle\frac{1}{\rho_{i-1}(H_{i-1}+\rho_{i-1})}-\displaystyle\frac{1}{\rho_i(H_i+\rho_i)}\right)\\+\displaystyle\frac{r}{2}\det(\nu, \nu_d)\displaystyle\frac{1}{\rho_d(H_d+\rho_d)}+(\al\beta_k-\beta\al_k)\displaystyle\frac{r}{\rho_{k-1}^3}.
\end{multline}

WLOG we assume one of the unbounded edges of $P$ is the $x_1$-axis, then $\nu_1=(0, 1)$. Consider the interior normals $\nu_i$ satisfying $\det(\nu_{i-1}, \nu_i)=-1$, $i=2, \cdots, d$, we obtain $\al_i>0$, $i=2, \cdots, d$ and $\displaystyle\frac{\beta_d}{\al_d}<\displaystyle\frac{\beta_{d-1}}{\al_{d-1}}<\cdots<\displaystyle\frac{\beta_2}{\al_2}$. The condition (\ref{conditions on nu}) implies  

$$\al\beta_k-\beta\al_k\ge0\text{ for }k\in I.$$

Thus by comparing with \cite{AS12} Equation~(7), we know the above additional term is non-negative. Combining (\ref{AS's matrix}), (\ref{component for cusp edge with normal edge}), (\ref{component within cusp edges}), and (\ref{component for nu}), we conclude $\det(D\xi)>0$. 

\medskip

Next, we prove that $x=(x_1, x_2)$ define global symplectic action coordinates on $P\backslash\ell_I$ and at the same time $u(x)$ has the desired boundary behavior on $\partial P\backslash\ell_I$. Note for 

$$\xi=\displaystyle\frac{1}{2\displaystyle\sqrt{(H+a)^2+r^2}},$$

the primitive of $\ep=r\left(\displaystyle\frac{\partial\xi}{\partial r}dH-\displaystyle\frac{\partial\xi}{\partial H}dr\right)$, up to constants, is given by

$$\displaystyle\frac{1}{2}\left(1-\displaystyle\frac{H+a}{\displaystyle\sqrt{(H+a)^2+r^2}}\right).$$

Then for $\nu=0$, up to constants, we have

$$x_1=\beta_1H+\displaystyle\frac{1}{2}\displaystyle\sum_{i=1}^{d-1}(\beta_{i+1}-\beta_i)(H_i-\rho_i)-\displaystyle\frac{1}{2}\displaystyle\sum_{k\in I}\Lambda_k\beta_k\left(1-\displaystyle\frac{H_{k-1}}{\rho_{k-1}}\right),$$

$$x_2=-\al_1H-\displaystyle\frac{1}{2}\displaystyle\sum_{i=1}^{d-1}(\al_{i+1}-\al_i)(H_i-\rho_i)+\displaystyle\frac{1}{2}\displaystyle\sum_{k\in I}\Lambda_k\al_k\left(1-\displaystyle\frac{H_{k-1}}{\rho_{k-1}}\right).$$

Then we see $x$ extends continuously to $r=0$ except at the points $(H, r)=(-a_k, 0)$ for $k\in I$. The point $(H, r)=(-a_k, 0)$ corresponds to the cusp edge $\ell_k$ for $k\in I$. Consider the behavior of $x$ in the intervals on the $H$-axis:

\begin{enumerate}[(i)]
    \item For $H>-a_1$, then $x_1=\beta_1H$, and $x_2=-\al_1H$;
    \item for $1\le j<i_1-1$, $-a_{j+1}<H<-a_j$, then $x_1=\beta_{j+1}H+\displaystyle\sum_{i=1}^ja_i(\beta_{i+1}-\beta_i)$, and $x_2=-\al_{j+1}H-\displaystyle\sum_{i=1}^ja_i(\al_{i+1}-\al_i)$;
    \item for $j=i_1-1$, $-a_{i_1+1}<H<-a_{i_1-1}=-a_{i_1}$, then $x_1=\beta_{i_1+1}H+\displaystyle\sum_{i=1}^{i_1}a_i(\beta_{i+1}-\beta_i)-\Lambda_{i_1}\beta_{i_1}$, and $x_2=-\al_{i_1+1}H-\displaystyle\sum_{i=1}^{i_1}a_i(\al_{i+1}-\al_i)+\Lambda_{i_1}\al_{i_1}$;
    \item similar calculations show for $2\le k\le m$, and $i_{k-1}+1\le j<i_k-1$, with $-a_{j+1}<H<-a_j$, we have $x_1=\beta_{j+1}H+\displaystyle\sum_{i=1}^ja_i(\beta_{i+1}-\beta_i)-\displaystyle\sum_{\ell=1}^{k-1}\Lambda_{i_{\ell}}\beta_{i_{\ell}}$, and $x_2=-\al_{j+1}H-\displaystyle\sum_{i=1}^ja_i(\al_{i+1}-\al_i)+\displaystyle\sum_{\ell=1}^{k-1}\Lambda_{i_{\ell}}\al_{i_{\ell}}$; for $j=i_k-1$, with $-a_{i_k+1}<H<-a_{i_k-1}=-a_{i_k}$, we have $x_1=\beta_{i_k+1}H+\displaystyle\sum_{i=1}^{i_k}a_i(\beta_{i+1}-\beta_i)-\displaystyle\sum_{\ell=1}^{k}\Lambda_{i_{\ell}}\beta_{i_{\ell}}$, and $x_2=-\al_{i_k+1}H-\displaystyle\sum_{i=1}^{i_k}a_i(\al_{i+1}-\al_i)+\displaystyle\sum_{\ell=1}^{k}\Lambda_{i_{\ell}}\al_{i_{\ell}}$;
    \item for $H<-a_{d-1}$, then $x_1=\beta_dH+\displaystyle\sum_{i=1}^{d-1}a_i(\beta_{i+1}-\beta_i)-\displaystyle\sum_{\ell=1}^{m}\Lambda_{i_{\ell}}\beta_{i_{\ell}}$, and $x_2=-\al_{d}H-\displaystyle\sum_{i=1}^{d-1}a_i(\al_{i+1}-\al_i)+\displaystyle\sum_{\ell=1}^{m}\Lambda_{i_{\ell}}\al_{i_{\ell}}$.
\end{enumerate}

We want to show the following:

\begin{enumerate}
    \item $x=(x_1, x_2)$ gives a proper homeomorphism

    \begin{equation}\label{proper homeomorphism}
    x (H, 0): \partial\h\backslash\bigcup_{k\in I}(a_{k-1}, 0)\rightarrow\partial P\backslash\ell_I.
    \end{equation}
    \item $u(x)$ modulo a smooth function is given by

     \begin{equation}\label{global boundary behavior}
         \displaystyle\frac{1}{2}\displaystyle\sum_{i\notin I}\ell_i(x)\log\ell_i(x)+\displaystyle\frac{1}{2}\displaystyle\sum_{k\in I}(\det(\nu_{k+1}, \nu_{k-1})\ell_k(x)-\Lambda_k)\log\ell_k(x).
     \end{equation}

\end{enumerate}

 First, we discuss the choice of $a_j$. Note from \cite{AS12} Theorem~1.2, there exist real numbers $a_1'<a_2'<\cdots<a_{d-1}'$ determined by $P$ such that for $1\le j\le d-1$, 

\begin{equation}\label{aj}
    \displaystyle\sum_{i=1}^ja_i'\det(\nu_{i+1}-\nu_i, \nu_{j+1})=\lambda_{j+1}.
\end{equation}

More precisely, from \cite{SD21} Lemma~7.2, we know $a_i'-a_{i-1}'=\displaystyle\frac{L_i}{2\pi|\nu_i|^2}$. Here $L_i$ is the length of the $i$th edge of $P$. On the other hand, from the boundary behavior of the non-cusp edges in (\ref{global boundary behavior}), with the same proof as \cite{SD21} Lemma~7.2, we have

$$a_i-a_{i-1}=\displaystyle\frac{L_i}{2\pi|\nu_i|^2}.$$

For $k\in I$, we set $a_k=a_{k-1}$. Later we will see that
 this ensures the desired boundary behavior. Then we obtain the following relations between $a_j$ and $a_j'$:

\begin{equation}\label{aj aj' relation}
    a_{j}=a_{j}' \text{ if }j<i_1; a_{j}=a_{j}'+\displaystyle\sum_{\ell=1}^k(a_{i_{\ell}-1}'-a_{i_{\ell}}')\text { if } i_k\le j<i_{k+1}; a_{j}=a_{j}'+\displaystyle\sum_{\ell=1}^m(a_{i_{\ell}-1}'-a_{i_{\ell}}')\text { if } i_{m}\le j<d.
\end{equation}

 First, we look at condition (\ref{proper homeomorphism}). For simplicity, we write $k=i_1$. Near the edge $\ell_{k+1}$, as $r=0$ and $-a_{k+1}<H<-a_k$, we have 
 
 $$x_1=\beta_{k+1}H+\displaystyle\sum_{i=1}^ka_i(\beta_{i+1}-\beta_i)-\Lambda_k\beta_k, \quad x_2=-\al_{k+1}H-\displaystyle\sum_{i=1}^ka_i(\al_{i+1}-\al_i)+\Lambda_k\al_k.$$
 
 Then condition~(\ref{proper homeomorphism}) $\ell_{k+1}(x)=0$ translates to

\begin{equation}\label{p1q1condtion1}
    \displaystyle\sum_{i=1}^ka_i\det(\nu_{i+1}-\nu_i, \nu_{k+1})+\Lambda_k=\lambda_{k+1}.
\end{equation}

Using (\ref{aj}) and the relation (\ref{aj aj' relation}), it simplifies to 

\begin{equation}\label{simplified p1q1condtion1}
    \Lambda_k=a_k'-a_{k-1}'.
\end{equation}

For edges $\ell_j$ with $i<k$, the argument is the same as the standard case as in \cite{AS12} Theorem 4.1; for $i=k+j$, with $1\le j< i_2-k$, the condition (\ref{proper homeomorphism}) $\ell_{k+j}(x)=0$ on $-a_{k+j}<H<a_{k+j-1}$ becomes

$$\displaystyle\sum_{i=1}^{k+j-1}a_i\det(\nu_{i+1}-\nu_i, \nu_{k+j})+\Lambda_k\det(\nu_{k+j}, \nu_{k})=\lambda_{k+j}.$$

Using (\ref{aj}) it suffices to check $\det(\nu_{k+j}, \nu_k)(\Lambda_k+a_{k-1}'-a_k')=0$, which holds from (\ref{simplified p1q1condtion1}). Thus we've shown condition~(\ref{global boundary behavior}) for $1\le j<i_2$ given the choice of $\Lambda_k$ as in (\ref{simplified p1q1condtion1}). It remains to apply the same procedure to indices $i_2, \cdots, i_m$ respectively. For simplicity, we write $t=i_2$. For $-a_{t+1}<H<-a_t$, we have 

$$x_1=\beta_{t+1}H+\displaystyle\sum_{i=1}^ta_i(\beta_{i+1}-\beta_i)-\Lambda_t\beta_t-\Lambda_k\beta_k, \quad x_2=-\al_{t+1}H-\displaystyle\sum_{i=1}^ta_i(\al_{i+1}-\al_i)+\Lambda_t\al_t+\Lambda_k\al_k.$$

Then condition~(\ref{proper homeomorphism}) $\ell_{t+1}(x)=0$ on $-a_{t+1}<H<-a_t$ translates to

\begin{equation}
    \displaystyle\sum_{i=1}^ta_i\det(\nu_{i+1}-\nu_i, \nu_{t+1})+\Lambda_t+\Lambda_k\det(\nu_{t+1}, \nu_k)=\lambda_{t+1},
\end{equation}
 
and it simplifies to 

\begin{equation}\label{p2q2condition1}
    \Lambda_t=a_t'-a_{t-1}'.
\end{equation}

For $i_2<i<i_3$, condition (\ref{proper homeomorphism}) can be shown in exactly the same way. For $i_3, \cdots, i_m$, it's now clear this procedure also works. We deduce that given the choices of $\Lambda_i=a_i'-a_{i-1}'$ for all $i\in I$, $x:(\overline{\h}, \partial\h)\rightarrow(P, \partial P)$ is a proper homeomorphism, and its restriction to $\h$ is a smooth proper diffeomorphism onto $P^{\circ}$.

\medskip

Next, we look at the condition~(\ref{global boundary behavior}). Near a non-cusp edge $\ell_i$, we have $\xi=\nu_i\log r+O(1)$, and the boundary condition is standard as desired. We focus on the situation near the cusp edge $\ell_k$. For this, we prove the following lemmas:

\begin{lem}
    Assume for $k\in I$, $\Lambda_k=a_k'-a_{k-1}'$, then 

    \begin{equation}\label{decouple to 2 boundary condition01}
        \ell_k(x)=\rho_{k-1}+O(r^2),
    \end{equation}
    
    and there exist smooth positive function $\dt_{k-1, k, k+1}$ such that
    
   \begin{equation}\label{decouple to 2 boundary condition02}
    \rho_{k-1}+H_{k-1}=\rho_{k-1}\cdot\displaystyle\frac{2\ell_{k+1}(x)+O(r^2)}{\dt_{k-1, k, k+1}(x)},\quad \rho_{k-1}-H_{k-1}=\rho_{k-1}\cdot\displaystyle\frac{2\ell_{k-1}(x)+O(r^2)}{\dt_{k-1, k, k+1}(x)}.
    \end{equation}
\end{lem}

\begin{proof}

We first consider $k=i_1$. As both $H_{k-1}$ and $r$ go to zero, for $i>k$, we have $H_i-\displaystyle\sqrt{H_i^2+r^2}=-\displaystyle\frac{r^2}{2a_i}$, and for $i<k-1$, we have $H_i-\displaystyle\sqrt{H_i^2+r^2}=2a_i$. Then for the behavior of $x$ near $\ell_k(x)=0$, we write

$$x_1=\beta_{k-1}H+\displaystyle\frac{1}{2}(\beta_{k+1}-\beta_{k-1})(H_{k-1}-\rho_{k-1})-\displaystyle\frac{1}{2}\Lambda_k\beta_k\left(1+\displaystyle\frac{H_{k-1}}{\rho_{k-1}}\right)+\displaystyle\sum_{i=1}^{k-2}a_i(\beta_{i+1}-\beta_i)+O(r^2),$$

$$x_2=-\al_{k-1}H-\displaystyle\frac{1}{2}(\al_{k+1}-\al_{k-1})(H_{k-1}-\rho_{k-1})+\displaystyle\frac{1}{2}\Lambda_k\al_k\left(1+\displaystyle\frac{H_{k-1}}{\rho_{k-1}}\right)-\displaystyle\sum_{i=1}^{k-2}a_i(\al_{i+1}-\al_i)+O(r^2).$$ 

Then

$$\nu_k\cdot x=\rho_{k-1}-a_{k-1}+\displaystyle\sum_{i=1}^{k-2}a_i\det(\nu_k, \nu_{i+1}-\nu_i)+O(r^2).$$

To prove Equation~(\ref{decouple to 2 boundary condition01}), it's equivalent to show

\begin{equation}\label{p1q1condtion2}
    \displaystyle\sum_{i=1}^{k-2}a_i\det(\nu_k, \nu_{i+1}-\nu_i)-a_{k-1}+\lambda_k=0.
\end{equation}

From the relations between $a_j$ and $a_j'$ in Equation~(\ref{aj aj' relation}) and the expression of $\lambda_j$ in Equation~(\ref{aj}), direct computation shows this automatically holds.

\medskip

Similarly, we have

$$\nu_{k-1}\cdot x=\displaystyle\frac{1}{2}\det(\nu_{k+1}, \nu_{k-1})(\rho_{k-1}-H_{k-1})+\displaystyle\frac{\Lambda_k}{2}(1-\displaystyle\frac{H_{k-1}}{\rho_{k-1}})+\displaystyle\sum_{i=1}^{k-2}a_i\det(\nu_{k-1}, \nu_{i+1}-\nu_i)+O(r^2).$$

 Under the normalization assumption $\det(\nu_j, \nu_{j+1})=-1$ for all $j$, straightforward calculations give

\begin{equation}\label{three buddy}
   \det(\nu_{j+1}, \nu_{j-1})=\displaystyle\frac{\al_{j+1}+\al_{j-1}}{\al_j}=\displaystyle\frac{\beta_{j+1}+\beta_{j-1}}{\beta_j}. 
\end{equation}

Combining (\ref{three buddy}) and (\ref{decouple to 2 boundary condition01}) with the above calculations, we obtain 

$$H_{k-1}=\rho_{k-1}\cdot\displaystyle\frac{(\nu_{k+1}-\nu_{k-1})\cdot x+\lambda_{k}\det(\nu_{k+1}, \nu_{k-1})+2\displaystyle\sum_{i=1}^{k-2}a_i\det(\nu_{k-1}, \nu_{i+1}-\nu_i)+\Lambda_k+O(r^2)}{(\nu_{k+1}+\nu_{k-1})\cdot x+\lambda_{k}\det(\nu_{k+1}, \nu_{k-1})+\Lambda_k+O(r^2)}.$$

Thus to prove Equation~(\ref{decouple to 2 boundary condition02}), it suffices to show

\begin{equation}\label{p1q1condtion3}
    \lambda_{k+1}=\Lambda_k+\lambda_k\det(\nu_{k+1}, \nu_{k-1})+\displaystyle\sum_{i=1}^{k-2}a_i\det(\nu_{k-1}, \nu_{i+1}-\nu_i),
\end{equation}

and

\begin{equation}\label{p1q1condtion4}
    \lambda_{k-1}=-\displaystyle\sum_{i=1}^{k-2}a_i\det(\nu_{k-1}, \nu_{i+1}-\nu_i).
\end{equation}

We claim that

\begin{equation}\label{p1q1final}
    \lambda_{k+1}-\lambda_k\det(\nu_{k+1}, \nu_{k-1})+\lambda_{k-1}-\Lambda_k=0.
\end{equation}

Using (\ref{aj}), we have

$$\lambda_{k+1}+\lambda_{k-1}=a_{k}-a_{k-1}+a_{k-1}\det(\nu_{k+1}, \nu_{k-1})+\displaystyle\sum_{i=1}^{k-2}a_i\det(\nu_{k-1}+\nu_{k+1}, \nu_{i+1}-\nu_i),$$

then by simplifying this equation with (\ref{aj aj' relation}), we obtain (\ref{p1q1final}). Then (\ref{p1q1condtion3}) is reduced to (\ref{p1q1condtion4}), which is straightforward from (\ref{aj}). Then we obtain Equation~(\ref{decouple to 2 boundary condition02}) with 

$$\dt_{k-1, k, k+1}=\ell_{k+1}(x)+\ell_{k-1}(x)+O(r^2).$$

It remains to show Equations~(\ref{decouple to 2 boundary condition01}) and (\ref{decouple to 2 boundary condition02}) for $i_2, \cdots, i_m$. Write $t=i_2$, near $\ell_t(x)=0$, we write

$$x_1=\beta_{t-1}H+\displaystyle\frac{1}{2}(\beta_{t+1}-\beta_{t-1})(H_{t-1}-\rho_{t-1})-\displaystyle\frac{1}{2}\Lambda_t\beta_t\left(1+\displaystyle\frac{H_{t-1}}{\rho_{t-1}}\right)+\displaystyle\sum_{i=1}^{t-2}a_i(\beta_{i+1}-\beta_i)-\Lambda_k\beta_k+O(r^2),$$

$$x_2=-\al_{t-1}H-\displaystyle\frac{1}{2}(\al_{t+1}-\al_{t-1})(H_{t-1}-\rho_{t-1})+\displaystyle\frac{1}{2}\Lambda_t\al_t\left(1+\displaystyle\frac{H_{t-1}}{\rho_{t-1}}\right)-\displaystyle\sum_{i=1}^{t-2}a_i(\al_{i+1}-\al_i)-\Lambda_k\al_k+O(r^2).$$

Then Equation~(\ref{decouple to 2 boundary condition01}) holds automatically under the assumption $\Lambda_t=a_t'-a_{t-1}'$, and the first equation in Equation~(\ref{decouple to 2 boundary condition02}) simplifies to

\begin{equation}\label{p2q2condtion3}
    0=\lambda_{t+1}-\lambda_t\det(\nu_{t+1}, \nu_{t-1})+\lambda_{t-1}-\Lambda_t.
\end{equation}

Direct computation as the previous situation shows this equation holds. For $i_3, \cdots, i_m$, it's now clear this procedure also works. 
\end{proof}

\begin{lem}\label{lemma boundary behavior}
    There exists a smooth and strictly positive function $\dt$ on $P$ such that

\begin{equation}\label{boundary behavior}
        \det(\Hess_x(u))=\left(\dt\prod_{i=1}^d\ell_i\prod_{k\in I}\ell_k\right)^{-1}.
    \end{equation}
\end{lem}

\begin{proof}
    From \cite{donaldson2009constant}, we know $r=\left(\det \Hess_x(u)\right)^{-1/2}$. Then it suffices to show there exists a smooth and strictly positive function $\dt$, such that

    \begin{equation}\label{equivalent boundary behavior}
        \prod_{i=1}^d\ell_i\prod_{k\in I}\ell_k=\displaystyle\frac{r^2}{\dt}.
    \end{equation}
    
    From the above discussions, we know

    $$\displaystyle\frac{\partial x_1}{\partial r}=-\displaystyle\frac{r}{2}\displaystyle\sum_{i=1}^{d-1}\displaystyle\frac{\beta_{i+1}-\beta_i}{\rho_i}-\displaystyle\frac{r}{2}\displaystyle\sum_{k\in I}\Lambda_k\beta_k\displaystyle\frac{H_{k-1}}{\rho_{k-1}^3},\quad \displaystyle\frac{\partial x_2}{\partial r}=\displaystyle\frac{r}{2}\displaystyle\sum_{i=1}^{d-1}\displaystyle\frac{\al_{i+1}-\al_i}{\rho_i}+\displaystyle\frac{r}{2}\displaystyle\sum_{k\in I}\Lambda_k\al_k\displaystyle\frac{H_{k-1}}{\rho_{k-1}^3}.$$

    Then 

    \begin{align}
        \displaystyle\frac{\partial \ell_j}{\partial r}&=\displaystyle\frac{\partial x_1}{\partial r}\al_j+\displaystyle\frac{\partial x_2}{\partial r}\beta_j\\
        &=\displaystyle\frac{r}{2}\displaystyle\sum_{i=1}^{d-1}\displaystyle\frac{\det(\nu_{i+1}-\nu_i, \nu_j)}{\rho_i}+\displaystyle\frac{r}{2}\displaystyle\sum_{k\in I}\displaystyle\frac{\Lambda_kH_{k-1}}{\rho_{k-1}^3}(\al_k\beta_j-\al_j\beta_k)\\
        &=\displaystyle\frac{r}{2}\left(-\displaystyle\frac{\det(\nu_1, \nu_j)}{\rho_1}+\displaystyle\sum_{i=2}^{d-1}\det(\nu_i, \nu_j)\left(\displaystyle\frac{1}{\rho_{i-1}}-\displaystyle\frac{1}{\rho_{i}}\right)+\displaystyle\frac{\det(\nu_d, \nu_j)}{\rho_{d-1}}\right)+\displaystyle\frac{r}{2}\displaystyle\sum_{k\in I}\displaystyle\frac{\Lambda_kH_{k-1}}{\rho_{k-1}^3}\det(\nu_k, \nu_j).\label{partial ell j partial r}
    \end{align}

    We want to show for $j\notin I$, as we approach each edge $E_j$ of $P$, $\displaystyle\frac{\partial\ell_j}{\partial r}=r\dt_j$ for some smooth and positive function $\dt_j$; for $j\in I$, as $H$ approaches $-a_j$ and $r$ approaches $0$, $\ell_j^2\ell_{j-1}\ell_{j+1}=r^2\dt_j$ for some smooth and positive function $\dt_j$. To show these, we have the following discussions:


    \begin{enumerate}[(i)]
        \item When $j=1$, for $r=0$, $-a_1<H$, it's immediate to see each term of (\ref{partial ell j partial r}) is positive. This implies $\displaystyle\frac{\partial\ell_j}{\partial r}=r\dt_j$ for some function $\dt_j$ smooth and strictly positive.
        \item When $j>1$ and $j, j+1\notin I$. For $r=0$, $-a_{j}<H<-a_{j-1}$, given $1\le i<j$, we have $\det(\nu_i, \nu_j)<0, \displaystyle\frac{1}{\rho_{i-1}}-\displaystyle\frac{1}{\rho_{i}}<0$; given $j+1\le i<d$, we have $\det(\nu_i, \nu_j)>0, \displaystyle\frac{1}{\rho_{i-1}}-\displaystyle\frac{1}{\rho_{i}}>0$; given $k\in I, k<j$, we have $\det(\nu_k, \nu_j)<0, H_{k-1}<0$; given $k\in I, k>j$, we have $\det(\nu_k, \nu_j)>0, H_{k-1}>0$. Hence, each term of (\ref{partial ell j partial r}) is again positive.
        \item When $j=d$, for $r=0$, $H<-a_d$, similar arguments show each term is positive. 
        \item When $j+1\in I$, for $r=0$, $-a_j=-a_{j+1}<H<-a_{j-1}$, the only difference is when $i=j+1\in I$, we have $\rho_i=\rho_{i-1}$ and $\det(\nu_{i}, \nu_j)>0$, $H_j>0$, thus the involved terms are still positive. All other terms are still positive from the discussion in the second case.
        \item When $j\in I$. For $H=-a_j=-a_{j-1}$, $r$ doesn't take the value $0$, we let $r\rightarrow0$, then $\rho_j, \rho_{j-1}\rightarrow0$. In this case, $\displaystyle\frac{1}{\rho_{j-1}}$ goes to infinity, and we need a different argument. Using Equation~(\ref{decouple to 2 boundary condition02}), we obtain

        $$r^2=(\rho_{j-1}+H_{j-1})(\rho_{j-1}-H_{j-1})=\ell_j^2\cdot\displaystyle\frac{(2\ell_{j-1}+O(r^2))(2\ell_{j+1}+O(r^2))}{\dt_{j-1, j, j+1}^2},$$

        here $\dt_{j-1, j, j+1}$ is a smooth and positive function. Thus $\ell_j^2\ell_{j-1}\ell_{j+1}=r^2\dt_j$ for some smooth and positive function $\dt_j$.        
    \end{enumerate}
    
\end{proof}

With the preparation of these lemmas, we are ready to show the boundary condition~(\ref{global boundary behavior}) holds for the cusp edges of $P$. Consider $k\in I$, near $\ell_k(x)=0$, 

$$\xi_1=\al_{k-1}\log r+\displaystyle\frac{1}{2}(\al_{k+1}-\al_{k-1})\log(H_{k-1}+\rho_{k-1})-\displaystyle\frac{\Lambda_k}{2}\displaystyle\frac{\al_{k}}{\rho_{k-1}}+O(1),$$

$$\xi_2=\beta_{k-1}\log r+\displaystyle\frac{1}{2}(\beta_{k+1}-\beta_{k-1})\log(H_{k-1}+\rho_{k-1})-\displaystyle\frac{\Lambda_k}{2}\displaystyle\frac{\beta_{k}}{\rho_{k-1}}+O(1).$$

From Lemma~\ref{lemma boundary behavior}, we know $r^2=\dt'\ell_{k-1}\ell_k^2\ell_{k+1}$ for some smooth and strictly positive function $\dt'$ near $\ell_k$. Note from Equations~(\ref{decouple to 2 boundary condition01}) and (\ref{decouple to 2 boundary condition02}), we obtain from (\ref{three buddy}) that

\begin{multline}
    du=\displaystyle\frac{1}{2}\left(\log(\ell_{k-1})\al_{k-1}+\log(\ell_{k+1})\al_{k+1}+\left(\det(\nu_{k+1},\nu_{k-1})\log(\ell_k)-\displaystyle\frac{\Lambda_k}{\ell_{k}}\right)\al_k\right)dx_1\\
    +\displaystyle\frac{1}{2}\left(\log(\ell_{k-1})\beta_{k-1}+\log(\ell_{k+1})\beta_{k+1}+\left(\det(\nu_{k+1},\nu_{k-1})\log(\ell_k)-\displaystyle\frac{\Lambda_k}{\ell_{k}}\right)\beta_k\right)dx_2+O(1).
\end{multline}

This is the desired boundary behavior. Hence proving (\ref{global boundary behavior}). This completes the construction of the metric. The corresponding potential $u\in S_{\frac{\Lambda_k}{2}, \frac{1}{2}\det(\nu_{k-1}, \nu_{k+1})}(P, L, \ell_k)$ for $k\in I$ is immediate from the construction.
\end{proof}

\begin{rmk}
     We remark here that we can modify the assumption (\ref{non-adjacent index}) on the index set to $1<i_1\le i_2\le\cdots\le i_d<d$. If $k, k+1\in I$ and $k-1, k+2\notin I$, which means there are two adjacent cusp edges, then we have $a_{k-1}=a_k=a_{k+1}$. Now define for $i<k$, $\al'_i\coloneqq\al_i, \beta_i'\coloneqq\beta_i$; for $i=k$, $\al'_k\coloneqq\al_k+\al_{k+1}$, $\beta'_k\coloneqq\beta_k+\beta_{k+1}$; for $i>k$, $\al'_i\coloneqq\al_{i-1}, \beta_i'\coloneqq\beta_{i-1}$. Then, the situation is reduced to that discussed in the theorem. Similarly, we can extend the arguments to the case for any number of adjacent cusp edges with this argument. From this, we can relax the assumption of $D_i\cap D_j=\emptyset$ for all $i\ne j$ to allow some non-empty intersection of them.
\end{rmk}

We are interested in studying the asymptotic behavior of these Poincar\'e type scalar-flat K\"ahler metrics away from the divisor. For this, we compare them with the scalar-flat K\"ahler metrics on $\C^2$ discussed in \cite{donaldson2009constant} and \cite{weber2022generalized}. From \cite{weber2022generalized}, we know  the symplectic coordinates

$$x_1=\displaystyle\frac{1}{\sqrt{2}}(-r+\displaystyle\sqrt{H^2+r^2})+\al H^2, \quad x_2=\displaystyle\frac{1}{\sqrt{2}}(r+\displaystyle\sqrt{H^2+r^2})+\beta H^2, \text{ for } \al, \beta\ge 0$$

induce the Taub-NUT metric when $\al=\beta=0$, the generalized Taub-NUT metrics when $\al>0, \beta>0$, and the exceptional Taub-NUT metrics when $\al=0, \beta>0$ or $\beta=0, \al>0$.

\begin{thm}\label{theorem asymptotic behavior}
    Consider the metrics constructed in Theorem~\ref{Theorem unbounded PTK SFK}. The metric is (1) ALE if $\al=\beta=0$; (2) asymptotic to a generalized Taub-NUT metric if $\al, \beta>0$; (3) asymptotic to an exceptional Taub-NUT metric if either $\al=0, \beta>0$ or $\al>0, \beta=0$.
\end{thm}

\begin{proof}
    
From \cite{Don09}, Section 3, given angular coordinates $\theta_1, \theta_2$ for $x_1, x_2$, the metrics constructed in Theorem~\ref{Theorem unbounded PTK SFK} have the following form

$$\displaystyle\sum u_{ij}dx_i\otimes dx_j+\displaystyle\sum u^{ij}d\theta_i\otimes d\theta_j.$$

Equivalently $\displaystyle\sum u_{ij}dx_i\otimes dx_j$ can be written as

$$r\cdot \det D\xi(dH^2+dr^2).$$

For $H\ge 0$, and $\rho\coloneqq \displaystyle\sqrt{H^2+r^2}$, as $\rho\rightarrow\infty$, we have 

$$\displaystyle\frac{1}{\rho_i+H_i}-\displaystyle\frac{1}{\rho+H}=O\left(\displaystyle\frac{1}{\rho^2}\right), \displaystyle\frac{1}{\rho_i}-\displaystyle\frac{1}{\rho}=O\left(\displaystyle\frac{1}{\rho^2}\right),$$

then as $\rho\rightarrow\infty$, we rewrite $D\xi$ as 

\begin{equation}\label{det Dxi}
   D\xi=  \begin{pmatrix}
   \alpha&0\\\beta&0
\end{pmatrix}+\displaystyle\frac{1}{r}\begin{pmatrix}
    \displaystyle\frac{r\displaystyle\sum_{i=1}^{d-1}(\alpha_{i+1}-\alpha_i)}{2\rho}&\alpha_1+\displaystyle\frac{\displaystyle\sum_{i=1}^{d-1}(\alpha_{i+1}-\alpha_i)r^2}{2\rho(\rho+H)}\\\displaystyle\frac{r\displaystyle\sum_{i=1}^{d-1}(\beta_{i+1}-\beta_i)}{2\rho}&\beta_1+\displaystyle\frac{\displaystyle\sum_{i=1}^{d-1}(\beta_{i+1}-\beta_i)r^2}{2\rho(\rho+H)}
\end{pmatrix}+O\left(\displaystyle\frac{1}{\rho^2}\right).
\end{equation}

We see that compared to the Abreu and Sena-Dias case, the additional terms of the form $\displaystyle\frac{\al_k}{\rho_k}$ in $\xi$ only create terms of the form $O\left(\displaystyle\frac{1}{\rho^2}\right)$. Then under similar computations as in \cite{AS12} Proposition~5.1, we know

$$r\det D\xi=\displaystyle\frac{\det(\nu_d, \nu_1)}{2\rho}+O\left(\displaystyle\frac{1}{\rho^2}\right), \text { for }\al=\beta=0;$$

$$r\det D\xi=\det(\nu, \nu_1)\left(1-\displaystyle\frac{r^2}{2\rho(H+\rho)}\right)+\det(\nu, \nu_d)\displaystyle\frac{r^2}{2\rho(H+\rho)}+\displaystyle\frac{\det(\nu_d, \nu_1)}{2\rho}+O\left(\displaystyle\frac{1}{\rho^2}\right), \text{ for }(\al, \beta)\ne(0,0).$$

For $H\le 0$, we argue the same way by considering $(H, r)\mapsto (-H, r), (\nu_1, \cdots, \nu_d)\mapsto (\nu_d, \cdots, \nu_1)$, $(a_1, \cdots, a_d)\mapsto (-a_d, \cdots, a_1)$. Hence, we deduce the desired asymptotic behavior for the first two cases. The completeness of these metrics is immediate with the above calculations, for details, see \cite{AS12} Proposition~5.1. 

\medskip

For the last case where either $\al>0, \beta=0$ or $\al>0, \beta=0$, we compare $\det D\xi$ with that for the model exceptional Taub-NUT metrics on $\C^2$, introduced in \cite{weber2022generalized}. Let $\xi_{AS}$ be the Legendre transform of its momentum coordinate for the toric scalar-flat metric constructed in \cite{AS12} for the given polytope $P$, and $\xi_{exc}$ be the one for the exceptional Taub-NUT metric, then we know from \cite{weber2022asymptotic} Section 5 that $\det D\xi_{AS}=\det D\xi_{exc}+O(\rho^{-2})$. For our case, from the above expression (\ref{det Dxi}), we have $\det D\xi=\det D\xi_{AS}+O(\rho^{-2})$, and thus the metrics are asymptotic to the exceptional Taub-NUT metrics. 

\end{proof}

\begin{ex}(Hwang-Singer metric, \cite{Fen24}, \cite{FYZ16})\label{HS}
    Consider the polytope with three edges whose normal vectors are $\nu_1=(0, 1), \nu_2=(1,1), \nu_3=(1, 0)$ and $I=\{2\}$. Let
    
    $$L=\{x_1=0, x_2=0, x_1+x_2-1=0\}.$$
    
\begin{figure}[h]
\begin{center}
    \begin{tikzpicture}
       \filldraw[color=violet!1, fill=violet!8, ultra thin] (0, 2) -- (0, 0.4) --(0.4, 0) -- (2, 0);
       \filldraw[color=violet!60, fill=violet!8] (0, 2) -- (0, 0.4) (0.4, 0) -- (2, 0);
       \draw[color=violet!60, dashed] (0, 0.4)-- (0.4, 0);
    \end{tikzpicture}
    \caption{The moment polytope of the Hwang-Singer metric} 
\end{center}    
 \end{figure}
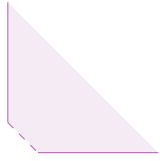 
 
 Then, from the construction, we obtain 

    $$\xi_1=\displaystyle\frac{1}{2}\log\left(H+\displaystyle\sqrt{H^2+r^2}\right)-\displaystyle\frac{1}{2}\displaystyle\frac{1}{\displaystyle\sqrt{H^2+r^2}}, \quad\xi_2=\displaystyle\frac{1}{2}\log\left(-H+\displaystyle\sqrt{H^2+r^2}\right)-\displaystyle\frac{1}{2}\displaystyle\frac{1}{\displaystyle\sqrt{H^2+r^2}};$$

    $$x_1=\displaystyle\frac{1}{2}\left(H+\displaystyle\sqrt{H^2+r^2}\right)+\displaystyle\frac{1}{2}\displaystyle\frac{H}{\displaystyle\sqrt{H^2+r^2}}, \quad x_2=\displaystyle\frac{1}{2}\left(-H+\displaystyle\sqrt{H^2+r^2}\right)-\displaystyle\frac{1}{2}\displaystyle\frac{H}{\displaystyle\sqrt{H^2+r^2}}.$$

    Then, we obtain the symplectic potential $u$, which is 

    $$u=\displaystyle\frac{1}{2}x_1\log x_1+\displaystyle\frac{1}{2}x_2\log x_2+\displaystyle\frac{1}{2}(x_1+x_2-1)\log(x_1+x_2)+h$$

    for some smooth function $h$. Let $\omega$ be the corresponding K\"ahler form. It lives on the complement of the zero section $E$ of the total space of the line bundle $\mathcal{O}(-1)$ over $\CP1$, which we denote by $Y$. Then $u\in S_{\frac{1}{2}, -\frac{1}{2}}(Y, L, E)$. Consider the momentum coordinate 

    $$\tau\coloneqq 2(x_1+x_2)=2\displaystyle\sqrt{H^2+r^2},$$

     let $X$ be the generator of the $S^1$-action satisfying $i_X\omega=-d\tau$, then from \cite{weber2023analytic} Section 2.3, we can compute the norm $\|X\|^2$ as follows:

     $$\|X\|^2=r\cdot\displaystyle\frac{\left(\displaystyle\frac{\partial\tau}{\partial H}\right)^2+\left(\displaystyle\frac{\partial\tau}{\partial r}\right)^2}{\displaystyle\frac{\partial\tau}{\partial H}\displaystyle\frac{\partial x_1}{\partial r}-\displaystyle\frac{\partial\tau}{\partial r}\displaystyle\frac{\partial x_1}{\partial H}}=\displaystyle\frac{2\tau^2}{2+\tau}.$$

     We see it is exactly the momentum profile of the Hwang-Singer metric discussed in \cite{Fen24} Section 2 and Section 3. If we vary the length of the edge corresponding to the divisor, then we obtain a one-parameter family of toric scalar-flat K\"ahler metrics of Poincar\'e type introduced in \cite{FYZ16}. More precisely, consider

     $$\xi_1=\displaystyle\frac{1}{2}\log\left(H+\displaystyle\sqrt{H^2+r^2}\right)-\displaystyle\frac{a}{2}\displaystyle\frac{1}{\displaystyle\sqrt{H^2+r^2}}, \quad\xi_2=\displaystyle\frac{1}{2}\log\left(-H+\displaystyle\sqrt{H^2+r^2}\right)-\displaystyle\frac{a}{2}\displaystyle\frac{1}{\displaystyle\sqrt{H^2+r^2}},$$

     the corresponding symplectic potential is given by 
     
     $$u=\displaystyle\frac{1}{2}x_1\log x_1+\displaystyle\frac{1}{2}x_2\log x_2+\displaystyle\frac{1}{2}(x_1+x_2-a)\log(x_1+x_2)+h_a$$
     
    for some smooth function $h_a$. Then $u\in S_{\frac{a}{2}, \frac{1}{2}}(Y, L, E)$. 
\end{ex}

With a similar approach, we can prove Theorem~\ref{theorem parallel edges}, again by explicitly constructing the toric metrics:

\begin{proof}[Proof of Theorem~\ref{theorem parallel edges}]
    Compared with the case in Theorem~\ref{Theorem unbounded PTK SFK}, all arguments of the construction work for the parallel edges case except for the choice of $\nu=(\al, \beta)$. Since $\nu_1=-\nu_d$, $\nu$ needs to satisfy

    $$\det(\nu, \nu_1)=0; \quad \det(\nu, \nu_k)\ge 0, \forall k\in I,$$

    then with the same arguments, we see $\xi_1, \xi_2$ give a one-parameter family of toric scalar-flat K\"ahler metrics. To understand the asymptotic behavior of these metrics, we compare them with that of the product metric $g_{prod}$ on $S^2\times\R^2$. On $S^2$, we take the round metric, and on $\R^2$, we take the hyperbolic metric. Then $S^2\times\R^2$ is biholomorphic to $\CP1\times D$. The symplectic coordinates of this product metric, as discussed in \cite{weber2023analytic}, can be written as

    $$dx_1=\displaystyle\frac{1}{2}\left(-1+\displaystyle\sqrt{H^2+r^2}+\displaystyle\sqrt{(H-1)^2+r^2}\right), \quad dx_2=\displaystyle\frac{1}{2}\left(1-\displaystyle\sqrt{H^2+r^2}+\displaystyle\sqrt{(H-1)^2+r^2}\right).$$   
    
    The toric scalar-flat metrics whose moment polytope has parallel unbounded edges are discussed by Weber in \cite{weber2023analytic} and \cite{weber2022generalized}. These metrics satisfy the Guillemin boundary condition. Let $\xi_{Web}$ be the Legendre coordinate of the momentum coordinate of the metric for the given polytope, and $\xi_{prod}$ be the one for the model metric. Then we know from \cite{weber2022asymptotic} Section 5 that $\det D\xi_{Web}=\det D\xi_{prod}+O(\rho^{-2})$. For our case, from the formula in Equation~(\ref{det Dxi}), we have $\det D\xi=\det D\xi_{Web}+O(\rho^{-2})$, thus conclude the asymptotic behavior of the constructed metrics.    
\end{proof}

At the end of this section, we remark that for a toric scalar-flat K\"ahler metric whose symplectic potential lives in $S_{\al, \beta}(P, F)$, the $\al, \beta$ are \textit{\textbf{uniquely determined}} by the polytope. We recall the following theorem in \cite{AAS17}:

\begin{thm}(\cite{AAS17}, Proposition 4.20)
    Consider $u\in S_{\al, \beta}(P, L, F)$ satisfies

    $$-\displaystyle\sum_{i, j}\displaystyle\frac{\partial^2u^{ij}}{\partial x_i\partial x_j}=s_{(P, L, F)}$$

    for some extremal affine linear function $s_{(P, L, F)}$ of $(P, L, F)$. Then $\al, \beta$ are determined by the data $(P, L, F)$.
\end{thm}

Following the proof of this theorem, an explicit expression of $\al$ and $\beta$ are obtained in Equations~(46) and (48). For our case where $s_{(P, L, F)}=0$, we see $\al, \beta$ are expressed in terms of functions on $F$, thus for non-compact $P$, they are still determined by the data $(P, L, F)$. Hence, for a fixed polytope $P$, if the symplectic potential of the toric scalar-flat K\"ahler metric lives in $S_{\al, \beta}$, then the choices of $\al, \beta$ must agree with those in Theorem~\ref{Theorem unbounded PTK SFK}. In the Appendix~\ref{Section}, we will discuss a uniqueness result under this prescribed class of symplectic potential.


\section{A conical family of toric metrics}\label{Section construction of conic}

Given a cone angle $2\pi\theta_0$ for $\theta_0\in(0, 1)$, motivated by the conical family in \cite{gauduchon2019invariant} Remark 1.2, we consider the following boundary behavior of potential $u$ of a toric metric which has conical singularity along the divisor corresponding to the edge $\ell(x)=0$ on its moment polytope:

\begin{equation}\label{boundary behavior of conical family}
    u(x)=\displaystyle\frac{1}{2\theta_0}\cdot\ell(x)\log\ell(x)+h_0(x)
\end{equation}

for some smooth function $h_0$. For the ALE scalar-flat K\"ahler metric of Poincar\'e type constructed in Theorem~\ref{Theorem unbounded PTK SFK}(i.e., $\nu=0$) whose moment polytope is $P\backslash\ell_I$. Let $u$ be its symplectic potential, write 

$$u=\displaystyle\frac{1}{2}\displaystyle\sum_{i\notin I}\ell_i\log\ell_i+\displaystyle\frac{1}{2}\displaystyle\sum_{i\in I}(\al_i+\beta_i\ell_i)\log\ell_i+h$$ 

for some $h\in C^{\infty}(P)$. Here $\al_i=-\displaystyle\frac{1}{2}\Lambda_i, \beta_i=\displaystyle\frac{1}{2}\det(\nu_{i+1}, \nu_{i-1})$ for $i\in I$. Consider the smooth scalar-flat K\"ahler metric on $P$ constructed by Abreu and Sena-Dias in \cite{AS12}, let $u_{AS}$ be the symplectic potential, we can write

$$u_{AS}=\displaystyle\frac{1}{2}\displaystyle\sum_{i=1}^{d-1}\ell_i\log\ell_i+h_{AS}$$ 

for some $h_{AS}\in C^{\infty}(P)$. For any $\theta=(\theta_{i_1}, \cdots, \theta_{i_m})$, consider
 
 \begin{equation}
 	u^{(\theta)}\coloneqq\displaystyle\frac{1}{2}\displaystyle\sum_{i\notin I}\ell_i\log\ell_i+\displaystyle\frac{1}{2}\displaystyle\sum_{i\in I}\displaystyle\frac{1}{\theta_i}\cdot\ell_i\log\ell_i+v_{\theta}+\displaystyle\prod_{i\in I}(1-\theta_i)\cdot h+\displaystyle\prod_{i\in I}\theta_i\cdot h_{AS}.
 \end{equation}


 Here $v_{\theta}=0$ for $\theta=1$ and for $\theta\in(0, 1)$, 

 $$v_{\theta}\coloneqq\displaystyle\sum_{i\in I}\Big[\left(\beta_i-\displaystyle\frac{1}{2\theta_i}\right)\ell_i\log\left(\ell_i+\displaystyle\frac{\theta_i}{1-\theta_i}\right)+\al_i\log\left(\ell_i+\displaystyle\frac{\theta_i}{1-\theta_i}\right)+\left(\beta_i-\displaystyle\frac{1}{2}\right)\ell_i\log\displaystyle\frac{1}{1-\theta_i}-\al_i\log\displaystyle\frac{1}{1-\theta_i}\Big].$$

Then $u^{(\theta)}\rightarrow u$ as $\theta_i\rightarrow0$ for all $i$ and $u^{(\theta)}\rightarrow u_{AS}$ as $\theta_i\rightarrow 1$ for all $i$. This family of conical metrics, however, are not necessarily scalar-flat. 

\medskip

In the remaining part of this section, we will follow the framework of Abreu and Sena-Dias to explicitly construct toric conical scalar-flat K\"ahler metrics:

\begin{thm} \label{Theorem unbounded conical SFK}

Given $X$ satisfying the same conditions as in Theorem~\ref{Theorem unbounded PTK SFK} and let $P$ be the moment polytope for $X$. Consider an index set $J=\{j_1, \cdots j_m\}\subset\{1, 2, \cdots, d\}$ with $1\le j_1<\cdots<j_m\le d$. For each $j_i\in J$, fix a cone angle $2\pi\theta_{j_i}$, we write $\theta=(\theta_{j_1}, \cdots \theta_{j_m})$. For the normal vector $\nu_i=(\al_i, \beta_i)$, consider

\begin{equation}\label{modified coefficients for conical family}
    \al_i'\coloneqq\al_i\text{ if }i\notin J, \al_i'\coloneqq\displaystyle\frac{\al_i}{\theta_i} \text{ if }i\in J; \quad 
    \beta_i'\coloneqq\beta_i\text{ if }i\notin J, \beta_i'\coloneqq\displaystyle\frac{\beta_i}{\theta_i}\text{ if }i\in J.
\end{equation}

Let $\ell_{J}=\bigcup_{i=1}^m\ell_{j_i}$ be the union of edges $\ell_{j_i}$ indexed by the elements in $J$, here $\ell_{j_i}$ corresponds to $D_{i}, \forall i=1, \cdots m$. Consider $\nu=(\alpha, \beta)$ a vector in $\mathbb{R}^2$ satisfying (\ref{conditions on nu}). Set

$$\xi_1\coloneqq\alpha'_1\log r+\displaystyle\frac{1}{2}\displaystyle\sum_{i=1}^{d-1}(\alpha'_{i+1}-\alpha'_i)\log\left(H+a_i^{(\theta)}+\displaystyle\sqrt{(H+a_i^{(\theta)})^2+r^2}\right)+\alpha H,$$

$$\xi_2\coloneqq\beta'_1\log r+\displaystyle\frac{1}{2}\displaystyle\sum_{i=1}^{d-1}(\beta'_{i+1}-\beta'_i)\log\left(H+a_i^{(\theta)}+\displaystyle\sqrt{(H+a_i^{(\theta)})^2+r^2}\right)+\beta H,$$

where $a_1^{(\theta)}, \cdots, a_{d-1}^{(\theta)}$ are real numbers determined by $P$ and $\theta$. Let $x_1, x_2$ be the primitives of 

$$\epsilon_1=r\left(\displaystyle\frac{\partial\xi_2}{\partial r}dH-\displaystyle\frac{\partial\xi_2}{\partial H}dr\right), \quad \epsilon_2=-r\left(\displaystyle\frac{\partial\xi_1}{\partial r}dH-\displaystyle\frac{\partial\xi_1}{\partial H}dr\right).$$

Then they define the momentum action coordinates on $P^{\circ}$ of some conical toric scalar-flat K\"ahler metric on $X$ whose cone angle along $D_i$ is $2\pi\theta_{j_i}$. Its symplectic potential satisfies

$$du=\xi_1dx_1+\xi_2dx_2.$$ 
\end{thm}

\begin{proof}
    Firstly, the positivity of $\det D\xi$ can be proved with the same arguments as in Theorem~\ref{Theorem unbounded PTK SFK}, it is because $(\al_i', \beta_i')$ either equals to $(\al_i, \beta_i)$, or is rescaled by a positive constant $\displaystyle\frac{1}{\theta_i}$ from $(\al_i, \beta_i)$ along edges corresponding to conical divisors. Next, to check the boundary behavior of $u$, we note for $-a_j^{(\theta)}<H<-a_{j-1}^{(\theta)}$, we have

$$\xi_1=\al'_1\log r+\displaystyle\sum_{i=1}^{j-2}(\al'_{i+1}-\al'_i)\log r=\al'_j\log r+O(1), \quad \xi_2=\beta_j'\log r+O(1),$$

which gives

$$du=\log r(\al'_jdx_1+\beta'_jdx_2)+O(1).$$

We claim that we still have $r=(\dt\prod_{i=1}^d\ell_i)^{1/2}$ for some smooth and positive function $\dt$. The reason is that with a similar analysis as in Lemma~\ref{lemma boundary behavior}, we obtain 

$$\displaystyle\frac{\partial x}{\partial r}=-\displaystyle\frac{r}{2}\displaystyle\sum_{i=1}^{d-1}\displaystyle\frac{\beta'_{i+1}-\beta'_i}{\rho_i^{(\theta)}}, \quad \displaystyle\frac{\partial\ell_j}{\partial r}=\displaystyle\frac{r}{2}\displaystyle\sum_{i=1}^{d-1}\displaystyle\frac{\det(\nu'_{i+1}-\nu'_i, \nu_j)}{\rho_i^{(\theta)}},$$

here $\nu'_i\coloneqq(\al'_i, \beta'_i)$, $H_i^{(\theta)}\coloneqq H+a_i^{(\theta)}$ and $\rho_i^{(\theta)}\coloneqq\displaystyle\sqrt{(H_i^{(\theta)})^2+r^2}$. Again, since $\al_i', \beta_i'$ are rescaled from $\al_i, \beta_i$, we still have

$$\displaystyle\frac{r}{2}\left(-\displaystyle\frac{\det(\nu_1', \nu_j)}{\rho_1^{(\theta)}}+\displaystyle\sum_{i=2}^{d-1}\det(\nu_i', \nu_j)\left(\displaystyle\frac{1}{\rho_{i-1}^{(\theta)}}-\displaystyle\frac{1}{\rho_{i}^{(\theta)}}\right)+\displaystyle\frac{\det(\nu_d', \nu_j)}{\rho_{d-1}^{(\theta)}}\right)>0,$$

 and this proves the claim. Thus, $du=\log\ell_j(\al'_jdx_1+\beta'_jdx_2)+O(1)$, giving the desired boundary behavior.

\begin{lem}
    Given $j\in J$, 

    \begin{equation}\label{length of conical edge}
        a_{j}^{(\theta)}-a_{j-1}^{(\theta)}=\theta_j\cdot\displaystyle\frac{L_{j}}{2\pi|\nu_{j}|^2}.   
    \end{equation}
\end{lem}

\begin{proof}
    Note

    $$\omega|_{\ell_j}=r\displaystyle\frac{\partial\xi_2}{\partial r}dH\wedge d\theta_1-r\displaystyle\frac{\partial\xi_1}{\partial r}dH\wedge d\theta_2,$$

    then from $\xi=\displaystyle\frac{1}{\theta_j}\nu_j\log r+O(1)$ on $\ell_j$, we get $\omega|_{\ell_j}=\displaystyle\frac{|\nu_j|^2}{\theta_j}dH\wedge dt.$ Then 

    $$L_j=\displaystyle\int_{\ell_j}\omega=2\pi(a_j^{(\theta)}-a_{j-1}^{(\theta)})\cdot\displaystyle\frac{|\nu_j|^2}{\theta_j},$$

    concluding (\ref{length of conical edge}).
\end{proof}

Also we know for $j\notin J$, $a_{j}^{(\theta)}-a_{j-1}^{(\theta)}=a_{j}'-a_{j-1}'$. Then, we obtain the relation between $a_j^{(\theta)}$ and $a_j'$:

\begin{multline}\label{aj aj theta relation}
    a_{j}^{(\theta)}=a_{j}' \text{ if }j<j_1; \quad a_{j}^{(\theta)}=a_{j}'+\displaystyle\sum_{\ell=1}^k(1-\theta_{j_{\ell}})(a_{j_{\ell}-1}'-a_{j_{\ell}}')\text { if } j_k\le j<j_{k+1}; \\\text{and }a_{j}^{(\theta)}=a_{j}'+\displaystyle\sum_{\ell=1}^m(1-\theta_{j_{\ell}})(a_{j_{\ell}-1}'-a_{j_{\ell}}')\text { if } j_{m}\le j\le d.
\end{multline}

From the expression of $\xi_1, \xi_2$ we deduce the expression of $x_1, x_2$ as follows:

$$x_1=\beta_1'H+\displaystyle\frac{1}{2}\displaystyle\sum_{i=1}^{d-1}(\beta_{i+1}'-\beta_i')(H_i^{(\theta)}-\rho_i^{(\theta)}), \quad x_2=-\al_1'H-\displaystyle\frac{1}{2}\displaystyle\sum_{i=1}^{d-1}(\al_{i+1}'-\al_i')(H_i^{(\theta)}-\rho_i^{(\theta)}).$$

When $r=0$, we have

$$x_1=\beta_1'H+\displaystyle\frac{1}{2}\displaystyle\sum_{i=1}^{d-1}(\beta_{i+1}'-\beta_i')(H_i^{(\theta)}-|H_i^{(\theta)}|), \quad x_2=-\alpha_1'H-\displaystyle\frac{1}{2}\displaystyle\sum_{i=1}^{d-1}(\al_{i+1}'-\al_i')(H_i^{(\theta)}-|H_i^{(\theta)}|).$$

Equivalently, we know

\begin{enumerate}[(i)]
    \item If $-a_1^{(\theta)}<H$, then $x_1=\beta_1'H, \quad x_2=-\al_1'H$;
    \item if $-a_{j+1}^{(\theta)}<H<-a_j^{(\theta)}$, then $x_1=\beta_{j+1}'H+\displaystyle\sum_{i=1}^ja_i^{(\theta)}(\beta_{i+1}'-\beta_i'), \quad x_2=-\al_{i+1}'H-\displaystyle\sum_{i=1}^ja_i^{(\theta)}(\al_{i+1}'-\al_i')$;
    \item if $H<-a_{d-1}^{(\theta)}$, then $x_1=\beta_d'H-\displaystyle\sum_{i=1}^{d-1}a_i^{(\theta)}(\beta_{i+1}'-\beta_i'), \quad x_2=-\al_d'H-\displaystyle\sum_{i=1}^{d-1}a_i^{(\theta)}(\al_{i+1}'-\al_i').$
\end{enumerate}

We want to show $x'=(x_1', x_2')$, when restricted to $(H, 0)$, defines a global proper homeomorphism. Again, we first focus on its behavior along $\ell_j$ for $j_1\le j<j_2$. For simplicity, we write $k=j_1$, then

$$x_1=\beta_{k-1}'H+\displaystyle\frac{1}{2}(\beta_k'-\beta_{k-1}')(H_{k-1}^{(\theta)}-\rho_{k-1}^{(\theta)})+\displaystyle\sum_{i=1}^{k-2}a_i^{(\theta)}(\beta_{i+1}'-\beta_i')+O(r^2),$$

$$x_2=-\al_{k-1}'H-\displaystyle\frac{1}{2}(\al_k'-\al_{k-1}')(H_{k-1}^{(\theta)}-\rho_{k-1}^{(\theta)})-\displaystyle\sum_{i=1}^{k-2}a_i^{(\theta)}(\al_{i+1}'-\al_i')+O(r^2).$$

From (\ref{aj aj theta relation}), we know $\ell_{k}(x)=0$ holds. For $k<j<j_2$, note $\ell_j(x)=0$ is equivalent to 

$$\displaystyle\sum_{i=2}^{j-1}\det(\nu_i, \nu_j)(a_{i-1}^{(\theta)}-a_i^{(\theta)})-a_1^{(\theta)}\det(\nu_1, \nu_j)=\displaystyle\sum_{i=2}^{j-1}\det(\nu_i', \nu_j)(a_{i-1}'-a_i')-a_1'\det(\nu_1, \nu_j).$$

From (\ref{aj aj theta relation}), direct computation shows the above equation holds. With essentially the same arguments applied to $j_i\le j<j_{i+1}$ for $i\ge 2$, we see $\ell_j(x)=0$ holds. Hence $x$ defines a global proper homeomorphism as desired.
\end{proof}

For the asymptotic behavior of the conical metrics, note as $\rho\rightarrow\infty$,

$$r\det D\xi=\displaystyle\frac{\det(\nu_d', \nu_1')}{2\rho}+O\left(\displaystyle\frac{1}{\rho^2}\right), \text { for }\al=\beta=0;$$

$$r\det D\xi=\det(\nu, \nu_1')\left(1-\displaystyle\frac{r^2}{2\rho(H+\rho)}\right)+\det(\nu, \nu_d')\displaystyle\frac{r^2}{2\rho(H+\rho)}+\displaystyle\frac{\det(\nu_d', \nu_1')}{2\rho}+O\left(\displaystyle\frac{1}{\rho^2}\right), \text { for }(\al, \beta)\ne(0,0).$$

The arguments in Theorem~\ref{theorem asymptotic behavior} still work here, and hence the asymptotic behavior of conical metrics coincide with those for the cuspidal metrics with a given choice of $(\al, \beta)$.

\section{Appendix: Uniqueness of toric metrics under given boundary conditions}\label{Section}

\begin{thm}\label{theorem uniqueness at the end}
    Consider the same setting as in Theorem~\ref{Theorem unbounded PTK SFK}. Assume $g$ is a toric scalar-flat K\"ahler metric of Poincar\'e type and its symplectic potential $u$ satisfies the prescribed boundary behavior given in (\ref{target boundary behavior}), then $g$ can only be one of the metrics constructed in Theorem~\ref{Theorem unbounded PTK SFK}.
\end{thm}

\begin{proof}

The proof closely follows the arguments of Sena-Dias in \cite{SD21}, with the essential differences in Claim~\ref{claim for surjectivity}, Lemma~\ref{lemma mu_0=r^2f} and Lemma~\ref{lemma f1 constant}. 

\medskip

Starting from a scalar-flat K\"ahler metric on a symplectic $4$-manifold, Donaldson shows in \cite{Don09} Theorem 1 that each solution of the scalar-flat K\"ahler equation locally arises from axi-symmetric harmonic functions $\xi_1, \xi_2$ in the way described in Theorem~\ref{theorem local model}. Here $\xi_1, \xi_2$ are unique up to translation in the $H$ variable and addition of constants. More precisely, given $u(x)$ the symplectic potential of a scalar-flat K\"ahler metric, we consider $(H, r)$ satisfying

\begin{equation}\label{from x1, x2 to H, r}
    r=(\det\Hess u)^{-1/2}, \quad \displaystyle\frac{\partial H}{\partial x_1}=-\displaystyle\frac{u^{2j}}{r}\displaystyle\frac{\partial r}{\partial x_j}, \quad\displaystyle\frac{\partial H}{\partial x_2}=\displaystyle\frac{u^{1j}}{r}\displaystyle\frac{\partial r}{\partial x_j}.
\end{equation}

Then for the moment $P^{\circ}$ endowed with the metric $g_{poly}=u_{ij}\displaystyle\sum_{i, j=1}^2dx_i\otimes dx_j$, $r$ is harmonic and $H$ is its harmonic conjugate. The metric $g_{poly}$ induces a complex structure $J_{poly}$ via its Hodge star. We obtain a $J_{poly}$-holomorphic local coordinate on $P^{\circ}$, written as 

$$z\coloneqq H+ir.$$

The coordinates $(H, r)$, as functions of $x_1, x_2$, are known as the isothermal coordinates. Note the boundary behavior of $u$ is determined on $\partial P$. Then $r$ extends continuously to $\partial P\backslash\ell_I$, as does $H$. Thus, $z$ extends to $\partial P\backslash\ell_I$ as a continuous function, denoted by $\tilde{z}$. Note $r=0$ on $\partial P\backslash \ell_I$, then $\tilde{z}(\partial P\backslash\ell_I)\subset\partial\h\backslash\bigcup_{k\in I}(-a_k, 0)$. From the boundary behavior of the metric we know $\tilde{z}$ is a bijection from $\partial P\backslash \ell_I$ to $\partial\h\backslash\bigcup_{k\in I}(-a_k, 0)$.

\begin{lem}
    The map $z: P^{\circ}\rightarrow\h$ is a bijection.
\end{lem}

\begin{proof}
    The proof relies on the real sub-manifold associated with the symplectic $4$-manifold $X$. From the discussions in \cite{KL15} Theorem 6.7 and \cite{SD21}, we know since the moment map is proper, $X$ is symplectomorphic to the quotient of some complex plane $\C^d$ by a sub-torus of the standard torus, with $d$ being the number of edges of the moment map of $X$. Note complex conjugation descends to a function on $X$ and $D$. We denote its fixed point set(which are real submanifolds) of $X$ and $D$ by $X_{\R}$ and $D_{\R}$, respectively. The moment map 
    
    $$\phi: X\backslash D\rightarrow\mathbb{R}^2\backslash\bigcup_{k\in I}(-a_k, 0),$$
    
    when restricted to $X_{\mathbb{R}}\backslash D_{\R}$ is denoted by $\phi_{\mathbb{R}}$. It is a 4 to 1 branched cover with the branched set being $\phi^{-1}(\partial P\backslash\ell_I)$ and write $\phi^{-1}_{\mathbb{R}}(P^{\circ})=\bigcup_{j=0}^{3}P_j$ as a disjoint union of the open sets $P_j$. Let $g_{\mathbb{R}}$ be the induced metric on $X_{\mathbb{R}}\backslash D_{\R}$ and $P_0$, then $g_{poly}\coloneqq \phi_{\R_*}(g_{\R})=u_{ij}dx_i\otimes dx_j$. Let 
    
    $$(X_{\mathbb{R}}\backslash D_{\R})^{\circ}\rightarrow X_{\mathbb{R}}\backslash D_{\R}$$
    
    be the orientable double cover of $X_{\mathbb{R}}\backslash D_{\R}$ and $\phi^{\circ}_{\mathbb{R}}$ be the lifting of $\phi_{\mathbb{R}}$ to $(X_{\mathbb{R}}\backslash D_{\R})^{\circ}$. For each $P_j$, $k=0, 1, 2$, we write $P_j^0$ and $P_j^1$ as the pre-image under the double cover. Via the Hodge star operator, we obtain from the metric induced by $g_{\mathbb{R}}$ on $(X_{\mathbb{R}}\backslash D_{\R})^{\circ}$ a complex structure $J_{\mathbb{R}}$, whose pushforward under $\phi^{\circ}_{\mathbb{R}}$ defines a complex structure $J_{poly}$ on $P\backslash\ell_I$. 

    \begin{claims}\label{lemma take neighbourhood p}
         Given $w\in\partial P\backslash\ell_I$, and $p$ an element in $(\phi^{\circ}_{\mathbb{R}})^{-1}(w)\subset (X_{\mathbb{R}}\backslash D_{\R})^{\circ}$, then there is a neighbourhood $V_p$ of $p$ in $(X_{\mathbb{R}}\backslash D_{\R})^{\circ}$ such that $z\circ\phi^{\circ}_{\mathbb{R}}$ extends to $V_p$ as a holomorphic function for $J_{\mathbb{R}}$. 
    \end{claims}

    The proof is a lifting and extension argument for the harmonic function $r$ and the harmonic conjugate $H$ on $X\backslash D$, and this argument is essentially the same as in \cite{SD21} Lemma 6.2. 
    
    \medskip
    
    Now, we show the injectivity. For the holomorphic map $z$, we consider its degree; it suffices to show the degree is $1$. If we consider a point $w_0\in \partial\h\backslash\bigcup_{k\in I}(-a_k, 0)$ and let $w\in \partial P\backslash\ell_I$ be its pre-image of $\tilde{z}$. Fix $p$ in the closure of $P_0^0\cap P_1^0$. From Claim~\ref{lemma take neighbourhood p}, there exists an extension of $z$ to $V_p$. Assume $V_p$ is small enough to admit a complex chart $z: V_p\rightarrow\C$. Then following the arguments of the proof of injectivity in \cite{SD21} Section 6.1, we know for any $\ep>0$ such that $B_{\ep}\cap P\subset\phi_{\R}^{\circ}(V_p)$, there exists $\dt$ such that $z^{-1}(B_{\dt}(w_0))\subset B_{\ep}(w)$; furthermore, given a point $w_0'\in B_{\dt}(w_0)\cap\h$, we can enlarge the loop $\gamma$ enclosing all pre-images of $z$ so that it also encloses $w$, then as $w_0'$ tends to $w_0$. The number of pre-images of $w_0'$ given by the integral

    $$\displaystyle\int_{\gamma}\displaystyle\frac{\frac{dz}{ds}(s)ds}{z(s)-w_0'}$$

    equals to that of $w_0$, which is $1$. 

    \medskip

    Then we show the surjectivity. The different boundary behavior for the symplectic potential in our case compared to that in \cite{SD21} doesn't cause any essential difference to the arguments of the proof. Note $P$ is non-compact and admits non-trivial harmonic functions, from the uniformization theorem, we know there exists a holomorphic map $\kappa: P^{\circ}\rightarrow\h$. 

    \begin{claims}
        $\kappa$ extends as a homeomorphism to $P\backslash\ell_I\rightarrow\overline{\h}\backslash\bigcup_{k\in I}(-a_k, 0)$, and it's a bijection.
    \end{claims}

    \begin{proof}

       First, we show the map is extendable. The different boundary behavior for the symplectic potential in our case compared to that in \cite{SD21} doesn't cause any essential difference to the arguments of the proof, for details we refer the readers to Lemma 5.4. 

       \medskip

       To see the extension is bijective, we argue by contradiction. Assume it's not injective, let $w, v\in \partial P\backslash\ell_I$ such that $\tilde{\kappa}(w)=\tilde{\kappa}(v)$. Take $o\in P^{\circ}$ and consider the Jordan curve going through $\kappa(o)$ and $\tilde{\kappa}(w)=\tilde{\kappa}(v)$. Let $\mathcal{C}$ be the interior of this Jordan curve, then the segment $\mathcal{S}$ joining $w$ and $v$ satisfies $\tilde{\kappa}(\mathcal{S})\subset\partial A\cap\partial\h$, and thus $\tilde{\kappa}$ is constant on $\mathcal{S}$. In particular, we see $\mathcal{S}$ doesn't contain $\ell_k$ for any $k\in I$. Then, the same arguments as in \cite{SD21} give a contradiction. Similarly, we can prove that the inverse is also injective.
    
    \end{proof}
    
    We use this extension map as an auxiliary to show the surjectivity of $z$. Let $U\coloneqq z(P^{\circ})$, write

    $$\partial U=(\partial U\cap\partial\mathbb{H})\cup(\partial U\cap\mathbb{H}),$$

    then the surjectivity of $z$ is equivalent to $\partial U\cap\mathbb{H}=\emptyset$. Consider $z_{\kappa}\coloneqq z\circ\kappa^{-1}:\mathbb{H}\rightarrow\mathbb{H}$, it is a holomorphic, injective map which can be extended bijectively to $\partial \mathbb{H}\backslash\bigcup_{k\in I}(-a_k, 0)\rightarrow\partial \mathbb{H}\backslash\bigcup_{k\in I}(-a_k, 0)$. Consider
    
    $$f(w)\coloneqq \displaystyle\frac{1}{z_{\kappa}(-\frac{1}{w})}: \mathbb{H}\rightarrow \mathbb{H},$$ 
    
    as in \cite{SD21}, the same arguments show that it is holomorphic and can be extended to a holomorphic function on $\mathbb{C}^*$ with $0$ being an isolated singularity using the Schwarz reflection principle. Rewrite $U=z_{\kappa}(\mathbb{H})$, we claim that 

    \begin{claims}\label{claim for surjectivity}
        $\partial U\cap\mathbb{H}=\{\lim z_{\kappa}(w_{n_k}), (w_k) \text{ unbounded with }(w_{n_k}) \text{ being a subsequence}\}.$
    \end{claims}

    \begin{proof}
        For any point $z_{\infty}$ in the above set, we take a sequence $w_k\in \mathbb{H}$ such that $z_{\kappa}(w_k)\rightarrow z_{\infty}$. If the sequence is bounded, there is a convergent subsequence $w_{n_k}$ in $\overline{\mathbb{H}}$ converging to $w\in\overline{\mathbb{H}}$. We have the following possibilities:
        
        \begin{enumerate}[(i)]
            \item If $w\in \mathbb{H}$, then $\tilde{z}_{\kappa}(w)=z_{\kappa}(w)\in U$ but since $U$ is open, $z_{\infty}\notin\partial U$;
            \item if $w\in\partial\mathbb{H}\backslash\bigcup_{k\in I}(-a_k, 0)$, then $\tilde{z}_{\kappa}\in\partial\h\backslash\bigcup_{k\in I}(-a_k, 0)$, then $z_{\infty}\notin \mathbb{H}$; 
            \item if $w=(-a_k, 0)$ for some $k\in I$, consider $\kappa^{-1}(w_{n_k})\coloneqq p_{n_k}$, we have $z(p_{n_k})\rightarrow z_{\infty}\in\mathbb{H}$. From $z: P\backslash\ell_I\rightarrow\overline{\h}\backslash\bigcup_{k\in I}(-a_k, 0)$ is bijective we know $p_{n_k}\rightarrow z^{-1}(z_{\infty})\in P\backslash\ell_I$ and from $\kappa: P\backslash\ell_I\rightarrow\overline{\h}\backslash\bigcup_{k\in I}(-a_k, 0)$ is bijective we know $\kappa(p_{n_k})\rightarrow z_{\kappa}^{-1}(z_{\infty})\in\overline{\h}\backslash\bigcup_{k\in I}(-a_k, 0)$. This implies $w_{n_k}$ tends to an element in $\overline{\h}\backslash\bigcup_{k\in I}(-a_k, 0)$, a contradiction.
        \end{enumerate} 
    \end{proof}

     Equivalently we know 

    $$\partial U\cap\mathbb{H}=\{\lim f(w_{n_k}), w_k\in\mathbb{H}, w_k\rightarrow 0, (w_{n_k}) \text{ a subsequence}\}\coloneqq f(0).$$

   Now, this set contains either $\infty$ or a single point, which is a pole. If the latter holds, we know $z_{\kappa}^{-1}$ is a holomorphic function with an isolated pole, but the image can not lie in $\mathbb{H}$. Hence we conclude that $\partial U\cap\mathbb{H}=\emptyset$. This concludes the proof.
\end{proof}

Now we know $z$ is a bijection, define $\mu\coloneqq z^{-1}$, and let $\mu_{ALE}\coloneqq z_{ALE}^{-1}$ be the corresponding map for the ALE metric $g_{ALE}$(i.e., the case where $\nu=0$) constructed in Theorem~\ref{Theorem unbounded PTK SFK}. 

\begin{lem}\label{lemma mu_0=r^2f}

Consider $\mu_0\coloneqq\mu-\mu_{ALE}$, we have $\mu_0=r^2 f$ for some $f\in C^{\infty}(\overline{\mathbb{H}}\backslash\bigcup_{k\in I}(-a_k, 0))$, and $f$ satisfies 

$$f_{HH}+f_{rr}+\frac{3f_r}{r}=0.$$
\end{lem}

\begin{proof}
    Consider 
    
    $$\eta\coloneqq (u_{x_1}, u_{x_2}), \quad\eta_{ALE}\coloneqq (u_{ALE, x_1}, u_{ALE, x_2}),$$
    
    we write
    
    $$\xi(H,r)=\eta\circ\mu(H,r), \quad \xi_0(H,r)=\eta\circ\mu_0(H,r).$$
    
    To show $\mu_0$ extends as an analytic function to $\overline{\mathbb{H}}\backslash\bigcup_{k\in I}(-a_k, 0)$, it's equivalent to show $\xi_0$ extends as an analytic function on $\overline{\mathbb{H}}\backslash\bigcup_{k\in I}(-a_k, 0)$. Since $\xi_0$ is an axi-symmetric harmonic function on $\h$, from the mean value theorem, it is sufficient to show $\xi_0$ is bounded in a neighborhood of each point on $\partial\mathbb{H}\backslash\bigcup_{k\in I}(-a_k, 0)$. Write

    $$\xi_0=\eta_{ALE}\circ\mu_{ALE}-\eta\circ\mu=(\eta_{ALE}\circ\mu_{ALE}-\eta_{ALE}\circ\mu)+(\eta_{ALE}\circ\mu-\eta\circ\mu).$$

    Note $\eta_{ALE}-\eta\in C^{\infty}$, and $\mu$ extends as a continuous function on $\overline{\mathbb{H}}\backslash\bigcup_{k\in I}(-a_k, 0)$, we know the second term is bounded. For the first term, rewrite it as $\xi_{ALE}-\xi_{ALE}\circ(\mu^{-1}_{ALE}\circ\mu)$. Near $\partial\mathbb{H}\backslash\bigcup_{k\in I}(-a_k, 0)$, there is a singularity of $\xi_{ALE}$ with
    
    $$\xi_{ALE}=\nu_i\log r+O(1) \text{ on each internal }-a_{i+1}<H<-a_i\text{ and }r=0\text{ given }i\notin I.$$
    
    
    It's equivalent to show as $r\rightarrow 0$, $\log \displaystyle\frac{r}{r(\mu^{-1}_{ALE}\circ\mu)}$ is bounded. Composing with $\mu^{-1}$ it suffices to show 
         
    \begin{enumerate}[(i)]
    \item For the conjugate harmonic coordinate $H'$ for $z$, for $i+1\notin I$, $-a_{i+1}<H<-a_i\iff -a_{i+1}<H<-a_i$, and for $i+1\in I$, $H=-a_{i+1}\iff H'=-a_{i+1}$;
    \item $\log \displaystyle\frac{r\circ\mu^{-1}}{r\circ\mu^{-1}_{ALE}}$ is bounded as $r$ approaches $0$.
    \end{enumerate}

    The first claim follows from \cite{SD21} Theorem 6.2 and our choice of coefficients (\ref{aj aj' relation}). For the second claim, recall $r\circ\mu^{-1}=\text{det(Hess }u\circ\mu^{-1})^{-1/2}$, then it's equivalent to show that 
    
    $$\displaystyle\frac{(\text{det Hess }u\circ\mu^{-1})^{-1/2}}{(\text{det Hess }u_{ALE}\circ\mu^{-1})^{-1/2}}$$
    
    is bounded. This follows from the assumed boundary behavior of $u$ and $u_{ALE}$. Direct calculation gives $f_{HH}+f_{rr}+\frac{3f_r}{r}=0$, for details see \cite{SD21} Lemma 6.3.
\end{proof}

\begin{lem}\label{lemma f1 constant}
    For the normal vector $\nu_1=(1,0)$, consider $f_1\coloneqq f\cdot\nu_1$, then $f_1$ is a constant.
\end{lem}

\begin{proof}
    Since $f_1$ is harmonic, it suffices to show it is bounded. From $f=\displaystyle\frac{\mu_0}{r^2}$ and $\nu\cdot\nu_1\ge 0$ we deduce that
    $f_1\le\displaystyle\frac{\mu_{ALE}\cdot\nu_1}{r^2}$. From the expression of $x_1, x_2$ in Theorem~\ref{Theorem unbounded PTK SFK}, we know 
    
    $$|\mu_{ALE}(H,r)|\le C\displaystyle\sqrt{H^2+r^2+1}.$$
    
    As in \cite{SD21}, we view $f$ as a harmonic function in $\R^5$, which only depends on the coordinate $H$ and the distance to the $H$-axis, $r$. Then $\forall w\in \R^5$,
    
    $$f_1(w)\le\displaystyle\frac{C(|w|+1)}{r^2}.$$
    
    Let $\mathcal{A}$ denote the subset of $\R^5$ whose $H$ coordinate does not take values $-a_k$ for any $k\in I$ when the distance $r$ vanishes. For a fixed $z\in \mathcal{A}$, there exists $R_z>0$ large enough such that $\partial B(z, R_z)\subset\mathcal{A}$, and with the mean value theorem, we get $f_1(z)\lesssim\displaystyle\frac{1}{R^3}\displaystyle\int_{\partial B(z, R)}\displaystyle\frac{dw}{r^2}$. Then from \cite{SD21} Lemma 6.4, for $R$ large enough, direct calculations show $\displaystyle\int_{\partial B(z, R)}\displaystyle\frac{dw}{r^2}\lesssim R^2$, thus $f_1$ is bounded from above.
\end{proof}

Similarly, for the normal vector $\nu_2$, we can define $f_2\coloneqq f\cdot\nu_1$, then $f_2$ is a constant. Thus $f$ is a constant, which implies $\xi_0=H\cdot\nu$ for some constant vector. Hence when $\nu=0$, $g$ is the ALE metric constructed in Theorem~\ref{Theorem unbounded PTK SFK} and otherwise $g$ is the generalized or exceptional Taub-NUT ones constructed there.

\end{proof}

\endgroup
\bibliography{reference}

\begin{bibdiv}
\begin{biblist}

\bib{Abr98}{article}{
      author={Abreu, Miguel},
       title={K\"ahler geometry of toric varieties and extremal metrics},
        date={1998},
     journal={International Journal of Mathematics},
      volume={9},
      number={06},
       pages={641\ndash 651},
}

\bib{Abr01}{article}{
      author={Abreu, Miguel},
       title={K\"ahler metrics on toric orbifolds},
        date={2001},
     journal={Journal of Differential Geometry},
      volume={58},
      number={1},
       pages={151\ndash 187},
}

\bib{Abr03}{article}{
      author={Abreu, Miguel},
       title={Kahler geometry of toric manifolds in symplectic coordinates},
        date={2003},
     journal={Symplectic and Contact Topology: Interactions and Perspectives, Fields Institute Communications},
      volume={35},
}

\bib{AS12}{article}{
      author={Abreu, Miguel},
      author={Sena-Dias, Rosa},
       title={Scalar-flat k{\"a}hler metrics on non-compact symplectic toric 4-manifolds},
        date={2012},
     journal={Annals of Global Analysis and Geometry},
      volume={41},
      number={2},
       pages={209\ndash 239},
}

\bib{AAS17}{article}{
      author={Apostolov, Vestislav},
      author={Auvray, Hugues},
      author={Sektnan, Lars~Martin},
       title={Extremal k\"ahler poincar\'e type metrics on toric varieties},
        date={2021},
     journal={The Journal of Geometric Analysis},
      volume={31},
       pages={1223\ndash 1290},
}

\bib{apostolov2004hamiltonian}{article}{
      author={Apostolov, Vestislav},
      author={Calderbank, David~MJ},
      author={Gauduchon, Paul},
      author={T{\o}nnesen-Friedman, Christina~W},
       title={Hamiltonian 2-forms in k{\"a}hler geometry, ii global classification},
        date={2004},
     journal={Journal of Differential Geometry},
      volume={68},
      number={2},
       pages={277\ndash 345},
}

\bib{Auv17}{article}{
      author={Auvray, Hugues},
       title={The space of poincar\'e type k\"ahler metrics on the complement of a divisor},
        date={2017},
     journal={Journal f\"ur die reine und angewandte Mathematik (Crelles Journal)},
      volume={2017},
      number={722},
       pages={1\ndash 64},
}

\bib{Bry01}{article}{
      author={Bryant, Robert},
       title={Bochner-k\"ahler metrics},
        date={2001},
     journal={Journal of the American Mathematical Society},
      volume={14},
      number={3},
       pages={623\ndash 715},
}

\bib{CY80}{article}{
      author={Cheng, Shiu-Yuen},
      author={Yau, Shing-Tung},
       title={On the existence of a complete k\"ahler metric on non-compact complex manifolds and the regularity of fefferman's equation},
        date={1980},
     journal={Communications on Pure and Applied Mathematics},
      volume={33},
      number={4},
       pages={507\ndash 544},
}

\bib{delzant1988hamiltoniens}{article}{
      author={Delzant, Thomas},
       title={Hamiltoniens p{\'e}riodiques et images convexes de l'application moment},
        date={1988},
     journal={Bulletin de la Soci{\'e}t{\'e} math{\'e}matique de France},
      volume={116},
      number={3},
       pages={315\ndash 339},
}

\bib{Don09}{article}{
      author={Donaldson, Simon},
       title={A generalised joyce construction for a family of nonlinear partial differential equations},
        date={2009},
     journal={Journal of G{\"o}kova Geometry Topology},
      volume={3},
       pages={1\ndash 8},
}

\bib{donaldson2009constant}{article}{
      author={Donaldson, Simon~K},
       title={Constant scalar curvature metrics on toric surfaces},
        date={2009},
     journal={Geometric and Functional Analysis},
      volume={19},
       pages={83\ndash 136},
}

\bib{Fen24}{article}{
      author={Feng, Yueqing},
       title={A gluing construction of constant scalar curvature k\"ahler metrics of poincar\'e type},
        date={2024},
     journal={arXiv preprint arXiv:2405.11952},
}

\bib{FYZ16}{article}{
      author={Fu, Jixiang},
      author={Yau, Shing-Tung},
      author={Zhou, Wubin},
       title={On complete constant scalar curvature k{\"a}hler metrics with poincar{\'e}--mok--yau asymptotic property},
        date={2016},
     journal={Communications in Analysis and Geometry},
      volume={24},
      number={3},
       pages={521\ndash 557},
}

\bib{gauduchon2019invariant}{article}{
      author={Gauduchon, Paul},
       title={Invariant scalar-flat k{\"a}hler metrics on $\mathcal{O}(\ell)$},
        date={2019},
     journal={Bollettino dell'Unione Matematica Italiana},
      volume={12},
       pages={265\ndash 292},
}

\bib{gibbons1978gravitational}{article}{
      author={Gibbons, Gary~W},
      author={Hawking, Stephen~W},
       title={Gravitational multi-instantons},
        date={1978},
     journal={Physics Letters B},
      volume={78},
      number={4},
       pages={430\ndash 432},
}

\bib{Gue15}{article}{
      author={Guenancia, Henri},
       title={K{\"a}hler--einstein metrics: From cones to cusps},
        date={2020},
     journal={Journal f{\"u}r die reine und angewandte Mathematik (Crelles Journal)},
      volume={2020},
      number={759},
       pages={1\ndash 27},
}

\bib{Gui94}{article}{
      author={Guillemin, Victor},
       title={Kaehler structures on toric varieties},
        date={1994},
     journal={Journal of differential geometry},
      volume={40},
      number={2},
       pages={285\ndash 309},
}

\bib{Joy95}{article}{
      author={Joyce, Dominic~D},
       title={Explicit construction of self-dual 4-manifolds},
        date={1995},
     journal={Duke Mathematical Journal},
      volume={77},
      number={3},
       pages={519},
}

\bib{KL15}{article}{
      author={Karshon, Yael},
      author={Lerman, Eugene},
      author={others},
       title={Non-compact symplectic toric manifolds},
        date={2015},
     journal={SIGMA. Symmetry, Integrability and Geometry: Methods and Applications},
      volume={11},
       pages={055},
}

\bib{Kob84}{thesis}{
      author={Kobayashi, Ryoichi},
       title={K\"ahler-einstein metric on an open algebraic manifold},
        type={Ph.D. Thesis},
        date={1984},
}

\bib{kronheimer1993construction}{incollection}{
      author={Kronheimer, PB},
       title={The construction of ale spaces as hyper-kahler quotients},
        date={1993},
   booktitle={Euclidean quantum gravity},
   publisher={World Scientific},
       pages={539\ndash 557},
}

\bib{lebrun1991complete}{inproceedings}{
      author={LeBrun, Claude},
       title={Complete ricci-flat k{\"a}hler metrics on cn need not be flat},
        date={1991},
   booktitle={Proc. symp. pure math},
      volume={52},
       pages={297\ndash 304},
}

\bib{lebrun1991explicit}{article}{
      author={LeBrun, Claude},
       title={Explicit self-dual metrics on $\mathbb{C}\mathbb{P}^2\#\cdots\mathbb{C}\mathbb{P}^2$},
        date={1991},
     journal={Journal of Differential Geometry},
      volume={34},
      number={1},
       pages={223\ndash 253},
}

\bib{Sek18}{article}{
      author={Sektnan, Lars~Martin},
       title={Blowing up extremal poincar{\'e} type manifolds},
        date={2023},
     journal={Mathematical Research Letters},
      volume={30},
      number={1},
       pages={185\ndash 238},
}

\bib{SD21}{article}{
      author={Sena-Dias, Rosa},
       title={Uniqueness among scalar-flat k{\"a}hler metrics on non-compact toric 4-manifolds},
        date={2021},
     journal={Journal of the London Mathematical Society},
      volume={103},
      number={2},
       pages={372\ndash 397},
}

\bib{TY87}{incollection}{
      author={Tian, Gang},
      author={Yau, Shing-Tung},
       title={Existence of k\"ahler-einstein metrics on complete k\"ahler manifolds and their applications to algebraic geometry},
        date={1987},
   booktitle={Mathematical aspects of string theory},
   publisher={World Scientific},
       pages={574\ndash 628},
}

\bib{weber2022asymptotic}{article}{
      author={Weber, Brian},
       title={Asymptotic geometry of toric kahler instantons},
        date={2022},
     journal={arXiv preprint arXiv:2208.00997},
}

\bib{weber2022generalized}{article}{
      author={Weber, Brian},
       title={Generalized k{\"a}hler taub-nut metrics and two exceptional instantons},
        date={2022},
     journal={Communications in Analysis and Geometry},
      volume={30},
      number={7},
       pages={1575\ndash 1632},
}

\bib{weber2023analytic}{article}{
      author={Weber, Brian},
       title={Analytic classification of toric k{\"a}hler instanton metrics in dimension 4},
        date={2023},
     journal={Journal of Geometry},
      volume={114},
      number={3},
       pages={28},
}

\end{biblist}
\end{bibdiv}
\Addresses
\end{document}